\documentclass[12pt]{article}
\usepackage[a4paper]{geometry}
\usepackage{amsthm}

\usepackage{amssymb}
\usepackage{amsmath}
\usepackage{eufrak}

\newtheorem{theorem}{Theorem}

\newtheorem{remark}[theorem]{Remark} 

\newtheorem{corollary}[theorem]{Corollary}
\newtheorem{lemma}[theorem]{Lemma}

\newtheorem{question}{Question}
\newtheorem{proposition}[theorem]{Proposition}

\begin{document}
\def\F{{\mathbb F}}
\title{ From  braces to pre-Lie rings}
\author{ Aner Shalev,  Agata  Smoktunowicz}
\date{}
This arXiv submission was split into two parts. The second part begins at page 17 and the  abstract of the second part appears on page 17.

\maketitle
\begin{abstract} 

Let $A$ be a brace of cardinality $p^{n}$ where $p>n+1$ is prime and let $ann (p^{2})$ be the set of elements of additive order at most $p^{2}$ in this brace. We construct a pre-Lie ring  related to the brace $A/ann(p^{2})$.

In the case of strongly nilpotent braces of nilpotency index $k<p$ the brace $A/ann(p^{2})$ can be recovered by applying the construction of the group of flows to the resulting pre-Lie ring.
We do not know whether or not our construction is related to the group of flows when applied to braces which are not right nilpotent.
\end{abstract}

\section{Introduction}

Let $p$ be a prime.
In \cite{Shalev} finite analogs of Lazard's $p$-adic Lie rings of $p$-adic Lie groups were constructed.
We apply analogous methods  to the multiplicative groups of braces, and combine them with methods from \cite{Rump, passage2}.  

Braces were introduced in 2007 by Wolfgang Rump \cite{rump}, and they are a generalisation of Jacobson radical rings, with the two-sided braces being exactly the Jacobson radical rings.

It is an open question whether, for $p > n+1$, there is a one-to-one correspondence between left  nilpotent pre-Lie rings of cardinality $p^{n}$ and braces
of the same cardinality (see Question $1$ in \cite{passage2} and Question 20.92 in \cite{kurovka}). Such correspondence holds for right nilpotent braces for sufficiently large $p$, but it is
not known whether each not right nilpotent brace of cardinality $p^{n}$, where $p>n+1$, corresponds to a pre-Lie algebra \cite{kurovka, passage2}.

Note that an affirmative answer to this question would yield an extension of the classical Lazard correspondence between $p$-adic Lie groups and $p$-adic Lie rings to the correspondence between braces and pre-Lie rings. This would make it possible for braces of this cardinality to be investigated by pre-Lie and Lie algebra researchers who do not have experience with braces.  

One of the main motivations for investigating braces is their connections with set-theoretic solutions of the Yang-Baxter equation \cite{rump, cjo} (see also \cite{pent, doikou, doikou2} for some connections with mathematical physics); another motivation is the connections of braces with homological group theory, since braces are precisely bijective  1-cocycles of groups.

The theory of braces is also connected to algebraic number theory and its generalisations through the concept of Hopf-Galois extensions of abelian type \cite{Bachiller}. It was shown by Gateva-Ivanova in \cite{gateva} that braces are in one-to-one correspondence with involutive braided groups, a structure which is used to investigate set-theoretic solutions of the Yang-Baxter equation since 1999.   
On the other hand, connections between braces and pre-Lie algebras make it possible to find connections between braces and symmetric brace algebras \cite{bracealgebras}, previously unrelated concepts.

The origins of connections between braces and pre-Lie algebras go back to the fundamental paper by Rump \cite{Rump}. It was he who noticed that pre-Lie $k$-algebras and pre-Lie $k$-braces are analogous concepts, and he described a bijective correspondence between these two classes of objects for $k=\mathbb R$ (with some additional requirements) \cite{Rump}. Although it was not explicitly stated in the paper, one direction of this correspondence is obtained by taking the group of flows of the pre-Lie algebra. Furthermore on page 141 of \cite{Rump} Rump suggests that similar techniques could be used to obtain a passage from pre-Lie $\mathbb F_{p}$-algebras to $\mathbb F_{p}$-braces using Lazard's correspondence, where ${\mathbb F}_{p}$ is the unique field of order $p$ for some prime $p$. Such a correspondence was formally obtained in \cite{Lazard} and generalised in \cite{passage2}.  In \cite{Lazard, passage2} a new operation $\cdot $ was defined on a left and right nilpotent brace $A$ such that $(A, +,\cdot)$ is a left and right nilpotent pre-Lie ring. 

 %passes from finite pre-Lie algebras to finite braces. He used the exponential function on a pre-Lie algebra to define a %multiplication operation and showed that the obtained structure, together with the same addition as in the pre-Lie algebra, is a %brace.

%This construction can also be described using the group of flows developed in \cite{AG}; for more details, see \cite{passage}.

 Any brace of prime power order is left nilpotent \cite{rump}. For this reason the left nilpotency of pre-Lie ring is a natural assumption that one can not expect to drop; however, it is an open question as to whether or not the assumption of right nilpotency can be dropped. Namely, 
 it is an open question whether every brace of cardinality $p^{n}$ where $p>n+1$ is a prime number
can be obtained from some pre-Lie algebra by using the group of flows. The affirmative answer would imply a one-to-one correspondence between braces and pre-Lie rings of such cardinality.
% This is known to be true for right nilpotent braces for sufficiently large $p$ \cite{passage}, and it is also  known to be true for $%\mathbb R$-braces \cite{Rump}, where for $\mathbb R$-braces the correspondence is local. 
%In the case of $\mathbb R$-braces Rump used the differential and methods of algebraic geometry such as affine torsors, to pass %from $\mathbb R$-braces to pre-Lie rings.

 In the case of finite braces, Rump, on page 141 of \cite{Rump}, suggested using the Lazard correspondence and  to use the  $1$-cocycle obtained from the brace to obtain the 1-cocycle in the Lie ring. It is then known that Lie rings with bijective  1-cocycles are pre-Lie rings.  

However there are complications, since the additive group of a Lie ring and the additive group of the brace may not be identical in the case where the adjoint group of the brace is obtained by using the Lazard correspondence from this Lie ring. 

%Note that the Lazard Lie ring of the multiplicative group of braces which are obtained from a pre-Lie ring as in \cite{Rump} is %isomorphic to the Lie-ring whose additive group coincides with that of the  original brace.  This makes it possible to apply %Rump's suggested method. However, it is not known whether for any brace there is a Lie ring isomorphic to the Lazard Lie ring %whose additive group is the same as the additive group of the original brace.
 
In this paper we construct a pre-Lie ring associated to the brace $A/ann(p^{2})$. This pre-Lie ring is then used in the subsequent paper \cite{newSmok} to show that, if $A$ is a brace of cardinality $p^{n}$ with $p>n+1$, then the brace  $A/ann(p^{4})$ can be obtained as in \cite{Rump} (i.e., as the group of flows) from some left nilpotent pre-Lie ring with the same additive group as the brace $A/ann(p^{4})$. This gives an approximation, up to elements which have additive order $p^4$, to the above questions.
%Recall that the Lazard correspondence gives rise to a correspondence between $p$-groups %of nilpotency class less than $p$ and nilpotent Lie rings of the same class and order.
%The main tools used to achieve this are the Baker-Campbell-Hausdorff formula and its %inverse formulae. In \cite{AG} the  construction of the group of flows of a pre-Lie algebra %was developed, allowing, under some assumptions, a passage from pre-Lie algebras to %groups.  
\section{Background information}
Recall that a   {\em pre-Lie ring} $(A, +, \cdot)$ is  an abelian group $(A,+)$ with a binary operation $(x, y) \rightarrow  x\cdot y$ such that 
\[(x\cdot y)\cdot z -x\cdot (y\cdot z) = (y\cdot x)\cdot z - y\cdot (x\cdot z)\]
and $(x+y)\cdot z=x\cdot z+y\cdot z, x\cdot (y+z)=x\cdot y+x\cdot z,$
 for every $x,y,z\in A$. We say that a pre-Lie ring $A$  is {\em  nilpotent} or {\em strongly nilpotent}   if for some $n\in \mathbb N$ all products of $n$ elements in $A$ are zero. We say that $A$ is {\em left nilpotent} if for some $n$, we have $a_{1}\cdot (a_{2}\cdot( a_{3}\cdot (\cdots  a_{n})\cdots ))=0$ for all  $a_{1}, a_{2}, \ldots , a_{n}\in A$.

Braces were introduced by Rump in 2007 \cite{rump} to describe all involutive, non-degenerate set-theoretic solutions of the Yang-Baxter equation.

Recall that a set $A$ with binary operations $+$ and $* $ is a {\em  left brace} if $(A, +)$ is an abelian group and the following  version of distributivity combined with associativity holds.
  \[(a+b+a*b)* c=a* c+b* c+a* (b* c), \space  a* (b+c)=a* b+a* c,\]
for all $a, b, c\in A$; moreover  $(A, \circ )$ is a group, where we define $a\circ b=a+b+a* b$.
In what follows we will use the definition in terms of the operation `$\circ $' presented in \cite{cjo} (see \cite{rump}
for the original definition): a set $A$ with binary operations of addition $+$ and multiplication $\circ $ is a brace if $(A, +)$ is an abelian group, $(A, \circ )$ is a group and for every $a,b,c\in A$
\[a\circ (b+c)+a=a\circ b+a\circ c.\]
  All braces in this paper are left braces, and we will just call them braces. 
Let $(A, +, \circ )$ be a brace.  Recall that $I\subseteq A$ is an ideal in $A$ if for
$i,j\in I$, $a\in A$ we have $i+j\in I, i-j\in I, i*a, a*i\in I$ where $a*b=a \circ b-a-b$.  It follows that $A/I$ is a well defined brace \cite{BaPhd}. We will denote by
$+, \circ$ the addition and the multiplication in  the brace $A/I$, using the same notation as for addition and multiplication in the brace $A$. Elements of the brace $A/I$ will be
denoted by $[a]_{I}$ where $[a]_{I}$ is a coset with $a\in A$ being a representative of this coset. Recall that $[a]_{I}=\{a+i:i\in I\}$ is a
subset of the brace $A$. Note also that $[a]_{I}+[b]_{I}=[a+b]_{I}$ and $[a]_{I}*[b]_{I}=[a*b]_{I}$.

Let $ann(p^{i})$ denote the subset of $A$ consisting of elements whose additive order is
 $p^{j}$ where $j\leq i$. Recall that if $A$ is a brace of cardinality $p^{n}$, where $p$ is a prime number
 larger than $n+1$, then $ann(p^{i})$ is an ideal in $A$ for every $i$ by Lemma $17$, \cite{passage} (Lemma $4$, \cite{cl}).

\section{ Some results which will be used later}

For the convenience of the reader, in this section we recall  some results  from other papers, which will be used later.

By a result of Rump \cite{rump}, for a prime number $p$,  every brace of order $p^{n}$ is left nilpotent.
Recall that Rump introduced {\em left nilpotent}  and  {\em right nilpotent}  braces and radical chains
$A^{i+1}$ and $A^{(i+1)}$  for a left brace $A$,  where inductively $A^{i}$ consists of sums of elements $a*b$ with $a\in A, b\in A^{i-1}$, and $A^{(i)}$ consists of sums of elements $a*b$ with $a\in A^{(i-1)}, b\in A$, and 
$A=A^{1}=A^{(1)}$. A left brace $A$  is left nilpotent if  there is a number $n$ such that $A^{n}=0$. A left brace $A$ is right nilpotent  if  there is a number $n$ such that $A^{(n)}=0$.

We recall the definition of a strongly nilpotent brace (defined in \cite{Engel}).
 Define $A^{[1]}=A$ and let $A^{[i]}$ consist of sums of elements $a*b$, with
$a\in A^{[j]}, b\in A^{[i-j]}$ for all $i>j>0$. A left brace $A$ is {\em strongly nilpotent}  if  there is a number $n$ such that $A^{[n]}=0$.  
We recall Lemma 4.1 from \cite{Engel}:

\begin{lemma}\label{citeEngel}
Let $(A, +, \circ)$ be a left brace satisfying $A^{s}=0$ for some positive 
integer $s$. Let $a, b\in A$, and as usual define $a*b=a\circ b-a-b$.
Define inductively elements $d_{i}=d_{i}(a,b), d_{i}'=d_{i}'(a, b)$  as follows:
$d_{0}=a$, $d_{0}'=b$, and for $i \geq 0$ define $d_{i+1}=d_{i}+d_{i}'$ and $d_{i+1}'=d_{i}*d_{i}'$.
Then for every $c\in A$ we have
\[(a+b)*c=a*c+b*c+\sum _{i=0}^{2s} (-1)^{i+1}((d_{i}*d_{i}')*c-d_{i}*(d_{i}'*c)).\]
\end{lemma}

For a brace $A$, an element $a\in A$ and a natural number $n$, let $a^{\circ n}=a\circ \cdots \circ a$ denote the product of $n$ copies of $a$ under the operation $\circ $.
 We recall Lemma 14 from \cite{note}:
\begin{lemma}\label{citenote}
  Let $A$ be a left brace, let $a, b\in A$ and let $n$ be a positive integer. Then,
  $a^{\circ n} *b =  \sum_{i=1}^{n}{n\choose i}e_{i},$ where
  $e_{1} = a*b$ 
 and for each $i$,  $e_{i+1} =a*e_{i}$.
   Moreover, $a^{\circ n} = \sum_{i=1}^{n}{n\choose i}a_{i},$ where
  $a_{1} = a$ 
 and for each $i$,  $a_{i+1} =a*a_{i}$.
\end{lemma}  

 For a brace $A$, denote by $A^{\circ p^{i}}$  the subgroup of $(A, \circ )$ generated by the elements $a^{\circ p^{i}}$, where $a\in A$. 

\begin{remark}\label{555} A $p$-group is said to be regular if $( xy) ^{p} = x^{p}y^{p}z$, where $z$ is an element of the commutator subgroup of the subgroup generated by $x$ and $y$.  If  $G$ is a group of cardinality $p^{n}$ where $p$ is a prime number larger than $n$  then $G$ is a regular $p$-group (see \cite{Janko} { \it \&7. regular $p$-groups, page 98}). By Theorem 7.2 (c) from \cite{Janko} this implies that 
  \[A^{\circ p^{i}}=\{a^{\circ p_{i}}: \ \  a\in A\},\] 
provided that $(A, +, \circ )$ is a brace of cardinality $p^{n}$ where $p>n+1$ is a prime number.
  \end{remark} 
For a brace $A$ and a natural number $m$ we denote $mA=\{ma:a\in A\}.$
We recall Proposition $15$ from \cite{passage} (Lemma 2, \cite{cl}) and Lemma $4$ from \cite{passage2}:
\begin{lemma}\label{citepassage}
Let $i,n$ be natural numbers.
Let $A$ be a brace of cardinality $p^{n}$ for some prime number $p>n+1$.
Then  $p^{i}A$ is an   ideal in $A$ for each $i$.
Moreover $A^{\circ p^{i}}=p^{i}A.$
\end{lemma}
Let $m$ be a natural number. An integer $\xi $  is a primitive root modulo $m$ if every integer coprime to $m$ is
congruent to a power of $\xi $ modulo $m$. It is known that
there exists a primitive root modulo $p^j$ for every $j$ and every odd prime number $p$ (see for example Theorem 6.11 \cite{numbertheory}).
 \begin{lemma}\label{Engelxi}
  Let $p>2$ be a prime number.
Let $\xi=\gamma ^{p^{p-1}}$, where $\gamma $ is a primitive root modulo $p^{p}$. Then $\xi ^{p-1}\equiv 1 \mod p^{p}$.  Moreover,  $ \xi ^{j}$ is not congruent to $1$ modulo $p$ for any natural number $0<j<p-1$.
\end{lemma}
\section{The Braces $pA$}
Assume that $B$ is a brace which is both left nilpotent and right nilpotent; then by a result from \cite{Engel} it is strongly nilpotent. In other words,  there is $k$ such that the product of any $k$ elements, in any order,  is zero (where all products are under the operation $*$).  If $B^{[k]}=0$ and $B^{[k-1]}\neq 0$, then we say that $B$ is  strongly nilpotent of degree $k$, and that $k$ is the nilpotency index of $B$.
% Let $A$ be a brace of cardinality $p^{n}$ where $p$ is a prime number larger than $n+1$. Denote $pA=\{pa:a\in A\}$ where %$pa$ is the sum of $p$ copies of $a$.
%In this section we investigate some  properties of  the braces $pA$.
\begin{proposition}\label{b} Let $n$ be a natural number.
Let $A$ be a  brace of cardinality $p^{n}$ where $p$ is a prime number larger than $n+1$. Then $pA$ is a brace, and the product
of any $p-1$ elements of $pA$ is zero.
Therefore, $pA$ is a strongly nilpotent brace of nilpotency index not exceeding $p-1$.
Moreover, every product of any $i$ elements from the brace $pA$ and any number of elements from $A$ belongs to $p^{i}A$.
Hence the product of any $p-1$ elements from the brace $pA$ and any number of elements from the brace $A$ is zero. Moreover, $p^{p-1}A=0$.
\end{proposition}

\begin{proof}Let $A$ be a brace of cardinality $p^{n}$ for $p>n+1$. Then $p^{i}A$ is an ideal in $A$ for every $i$ (by Lemma \ref{citepassage}).
Moreover, if $a\in p^{i}A, b\in p^{j}A$ then $a*b\in p^{i+j}A$, since if $b=p^{j}c$ then $a*p^{j}c=p^{j}(a*c)\in p^{j}(p^{i}A)=p^{i+j}A$.
It follows that $pA$ is a nilpotent brace, such that the product of any $n+1$ elements in this brace is in $p^{n}A$ and hence it is zero.
Notice that $p^{n}A=0$ since $(A, +)$ is a group of cardinality $p^{n}$, so the additive order of each element in $A$ divides $p^{n}$.
Since, by assumption, $p>n+1$, it follows that products of any $p-1$ elements in it is zero.
 \end{proof}

  We will now present two results which are related to Lemma \ref{citeEngel}, which will be used later in the proof of Proposition \ref{12345}. We first introduce a set $W$.

   Let $W$ denote the set of all non-associative words in non-commuting variables $X,Y, Z$ with any distribution of brackets; moreover, $Z$ appears only once at the very end of each word, and both  $X$ and $Y$ appear at least once in each word; furthermore, elements from the set $\{X,Y\}$ appear at least three times in each word. For example, $(XZ)(XY)\notin W$, while $((XY)X)Z\in W$.
 For $w\in W$ let  $w\langle a, b,  c\rangle $ denote a specialisation of the word $w$ for $X=a$, $Y=b$, $Z=c$ and the multiplication in $w\langle a, b, c\rangle $ is the same as the  operation $*$ in the brace $A$ (recall that $a*b=a \circ b-a-b$). So for example if $w=(((XX)X)Y)Z$ then $w\langle a,b, c\rangle =(((a*a)*a)*b)*c$. Let $w$ be a word in $X$ and $Z$, then $w\langle a, c\rangle $ is  the specialisation of the word $w$ for $X=a$,  $Z=c$.
  So for example if $w=X(X(XZ))$ then $w\langle a, c\rangle =a*(a*(a*c))$.
\begin{lemma}\label{co}
  Fix a prime number $p$. Let $W$ be as above. Then there are integers $\beta _{w}$ for $w\in W$, such that only a finite number of them is non zero and the following holds:
 For each brace $(A, +, \circ )$  of cardinality $p^{n}$ with  $n<p-1$ and for each $a, b\in pA$, $c\in A$
 we have
\[(a+b)*c=a*c+b*c+ a*(b*c)-(a*b)*c+\sum_{w\in W}\beta _{w}w\langle a, b,   c\rangle .\] 
\end{lemma}
\begin{proof} By Lemma \ref{citeEngel}, we have the following formula which holds for any brace of cardinality $p^{n}$ with $n<p-1$: for every $a,b\in pA$ and $c\in A$ we have
\[(a+b)*c=a*c+b*c+\sum _{i=0}^{2(p-1)} (-1)^{i+1}((d_{i}*d_{i}')*c-d_{i}*(d_{i}'*c)),\]
 where
$d_{0}=a$, $d_{0}'=b$, and for $i \geq 1$ we have  $d_{i+1}=d_{i}+d_{i}'$ and $d_{i+1}'=d_{i}*d_{i}'$.
 By writing the first summand from our sum before the sum we get

\[(a+b)*c=a*c+b*c-(a*b)*c+a*(b*c)+\sum _{i=1}^{2(p-1)} (-1)^{i+1}((d_{i}*d_{i}')*c-d_{i}*(d_{i}'*c)).\]

We will only use this relation and the relations that products of any $p-1$ elements from the set $\{a,b\}$ and an element $c$ is zero and the relations $p^{p-1}a=0$ for each $a\in A$ (this holds by Proposition \ref{b}). We then apply these relations several times, so that on the right hand side we have a sum of some products of elements $a$, $b$ and $c$ (where $c$ appears only once at the end and $b$ and $a$ both appear in each product), and then we obtain the result which only depends on $p$, and which does not depend on $a,b$ or on the brace $A$. This process will terminate by the last assertion from Proposition \ref{b}.
\end{proof}

Next, we obtain the following corollary.

\begin{corollary}\label{777}
  Let $p$ be a prime number and let $m$ be a number. Let  $W'$ be the set of nonassociative words in variables $X, Z$ where $Z$ appears only once at the end of each word and $X$ appears at least twice in each word.
  Then there are integers $\gamma _{w}$  for $w\in W'$, such that only a finite number of them is non zero and the following holds:
 For each brace $(A, +, \circ )$  of cardinality $p^{n}$ with $p>n+1$ and for each $a\in pA$,  $c\in A$
 we have
\[(ma)*c=m(a*c)+\sum_{w\in W'}\gamma _{w}w\langle a,c\rangle .\] 
\end{corollary}
\begin{proof}
 We will proceed by induction on $m$. For $m=1$ the result is true. For
$m=2$ the result follows from Lemma \ref{co} applied for $b=a$.
Suppose that the result is true for some $m$, then by the inductive assumption
$(ma)*c=m(a*c)+\sum_{w\in W'}\gamma _{w}w\langle a,c\rangle .$

We will first prove a supporting Fact $1$.

{\bf Fact $1$}. We will show that $Y_{t}^{\alpha }\subseteq Y_{t}^{1}$ for each $t,\alpha $ where the notation for this is as follows. 
 Let $a_{1}, a_{2} \ldots , a_{l}$ be all elements from the set $ pA$ and  denote $X=\{a_{1}, a_{2}, \ldots , a_{l}\}$ (so  $X=pA$ as sets).
For a natural number $t$ let $X_{t}^{1}$ denote the set consisting of products, under the operation $*$,  of $t$ or more elements from the set $X$ and also possibly an element $c\in A$ at the end; and let $Y_{t}^{1}$ denote the set whose elements are (finite) sums of elements from $X_{t}^{1}$. Notice that $X_{t}^{1}, Y_{t}^{1}\subseteq p^{t}A$.

Define $X_{t}^{2}$ to be set consisting of elements $u$ defined as follows. 
Let $q_{1}, q_{2}, \ldots q_{s}$ be such that $c_{i}\in Y_{j_{i}}^{1}$ for each $i$, where $j_{1}, \ldots , j_{s}>0$ are some natural numbers.  Let $u$ be a product of elements $q_{1}, q_{2}, \ldots , q_{s}$,   in this order, with any distribution of brackets,
  and such that  $t\leq {j_{1}+j_{2}+\ldots +  j_{s}}$ (where $s$ is a natural number which can be different for diferent elements $u$) and also possibly element $c$ at the end. Moreover, we only consider $u$ in which element $c$ appears at most once, and if it appears then it appears at the end.
 Let $Y_{t}^{2}$ denote the set whose elements are (finite) sums of elements from $X_{t}^{2}$. 
 Notice that  $X_{t}^{2}, Y_{t}^{2}\subseteq p^{t}A$.

Continuing in this way we define $X_{t}^{\alpha }, Y_{t}^{\alpha }\subseteq p^{t}A$.
Notice that $X_{i}^{\alpha -1}\subseteq X_{i}^{\alpha }$ for each $\alpha $, since $X_{i}^{\alpha -1}\subseteq Y_{i}^{\alpha -1}\subseteq X_{i}^{\alpha }$ for each $\alpha $.

We will show that $Y_{t}^{\alpha +1}\subseteq Y_{t}^{\alpha }$, for each $\alpha ,t$,
by induction on $t$ in the reverse order. Notice that it suffices to show that   $X_{t}^{\alpha +1}\subseteq Y_{t}^{\alpha }.$

Observe that this is true  for $t\geq p$, since  $Y_{t}^{\alpha +1} \subseteq p^{ p}A=0.$  Suppose that  it holds for all numbers $t>j$ for some $j$. We need to show that $X_{j}^{\alpha +1}\subseteq Y_{j}^{\alpha },$ for each $\alpha $.
Let $u\in X_{j}^{\alpha +1}$. Then
$u$ is a product of  $q_{1}, q_{2}, \ldots q_{s},$  such that $q_{i}\in Y_{j_{i}}^{\alpha },$ for some integers $j_{1}, \ldots , j_{s}$, such that $j_{1}+\cdots +j_{s}\geq j$ (for some $s$).

Let $q_{i}=\sum_{k}d_{i,k}$ for some $d_{i,k}\in X_{j_{i}}^{\alpha }$. By applying Lemma \ref{co} several times we obtain that
$u$ equals a sum of all possible products of elements 
$d_{1,k_{1}}$, $d_{2,k_{2}}, \ldots , d_{s, k_{s}}$  (in this order and with any distribution of brackets) for various $k_{1}, \ldots ,k_{s}$ plus some element $q\in  Y_{j +1}^{\alpha +\gamma }$ for some sufficiently large $\gamma $.
Notice that $d_{i,k_{i}}\in X_{j_{i}}^{\alpha }$ for each $i$, so  any product of elements 
$d_{1,k_{1}}$, $d_{2,k_{2}}, \ldots , d_{s, k_{s}}$  (in this order and with any distribution of brackets) 
 is in $X_{j}^{\alpha }$.

We will now show that $X_{j}^{\alpha +1}\subseteq Y_{j}^{\alpha }$, for each $\alpha ,t$.
We know that
 $X_{j}^{\alpha +1}\subseteq Y_{j}^{\alpha }+Y_{j+1}^{\alpha +\gamma},$ for each $\alpha $ and for some $\gamma=\gamma (\alpha )$.

By the inductive induction on $j$, we have $Y_{j+1}^{\alpha +\gamma}\subseteq Y_{j+1}^{\alpha +\gamma -1}\subseteq Y_{j+1}^{\alpha +\gamma -2}\subseteq \cdots \subseteq Y_{j+1}^{\alpha }$.
Since $Y_{j+1}^{\alpha }\subseteq Y_{j}^{\alpha }$ we obtain $X_{j}^{\alpha +1}\subseteq Y_{j}^{\alpha }$ for each $\alpha $, as required.
This completes the inductive argument. Therefore, $X_{j}^{\alpha +1}\subseteq Y_{j}^{\alpha }$ for each $j,\alpha $.
Therefore, $Y_{t}^{\alpha }\subseteq Y_{t}^{\alpha-1}\subseteq \ldots \subseteq Y_{t}^{1}.$
This proves the Fact $1$.

We are now ready to prove our result.
To calculate $(m+1)*a$ we can apply Lemma \ref{co} for the same $a$ and for $b=ma$ and obtain:
\[((m+1)a)*c=(a+(ma))*c=a*c+(ma)*c+ a*((ma)*c)-(a*(ma))*c+\sum_{w\in W}\beta _{w}w\langle a, ma,   c\rangle .\]

We  can then apply the inductive assumption and Fact $1$  to terms which appear on the right hand side to deduce the conclusion.
\end{proof}

We will now introduce some notation which will be used in the proof of Proposition \ref{12345} below.

 $ $

{\bf Notation 1.} Let $A$ be a brace of cardinality $p^{n}$, where $p$ is a prime number such that $n+1<p$. Let  $+, \circ , *$ be the operations in this brace. Assume that $x,y\in pA, z\in A$, and
let $E(x, y, z)\subseteq A$ denote the set consisting of any formal product of elements $x$ and $y$ and one element $z$ at the end of each product under the operation $*$,  in any order, with any distribution of brackets,  each formal product having $x$ and $y$ appear at least once. Moreover, we assume that all formal products which are elements of $E(x,y,z)$ have length less than $p$.  Intuitively this assumption is justified by the fact  that all products of $p-1$ or more elements from $pA$ and an element from $A$ at the end  are $0$ by Proposition \ref{b}. Let $V_{x,y,z}$ be  a vector obtained from elements of $E(x,y,z)$ arranged in a such way that shorter formal  products of elements $x,y,z$  are situated  before longer formal  products.

% We consider only products of less than $p$ elements (so there are less than $p-1$ occurrences of elements from the set $\{x,y\}%$).

$ $ 

\section{ The binary operation $\cdot $}
 In this section we will prove the following proposition: 

\begin{proposition}\label{12345}
Let $A$ be a  brace of cardinality $p^{n}$, where $p>n+1$. Let $\xi =\gamma  ^{p^{p-1}}$, where $\gamma $ is a primitive root modulo $p^{p}$. Define the binary operation $\cdot $ on $A$ as follows.
\[a\cdot b=\sum_{i=0}^{p-2}\xi ^{p-1-i}((\xi ^{i}a)* b),\]
    for $a, b\in A$.
 Then $(a+b)\cdot c=a\cdot c+b\cdot c, a\cdot (b'+c)=a\cdot b'+a\cdot c$
and $(a\cdot b)\cdot c-a\cdot (b \cdot c)=(b\cdot a)\cdot c-b\cdot (a\cdot c),$
for every $a, b\in pA$, $b', c\in A$.
In particular, $(pA, +, \cdot )$ is a pre-Lie algebra.
\end{proposition}
\begin{proof} {\bf Part 1.} By the definition of a left brace $x\cdot ( y+ z)=x\cdot y+ x\cdot z$. We  will show that
$(x+y)\cdot z=x\cdot z+y\cdot z$ for $x,y\in pA, z\in A$.
The proof is very similar to the proof of Proposition $5$ from \cite{passage2}.
Observe that   $(x+y)\cdot z= \sum_{i=0}^{p-2}\xi ^{p-1-i}((\xi ^{i}x+\xi ^{i}y)* z).$
When we apply  Lemma \ref{co}  to $a=\xi ^{i} x$ and $b= {\xi ^{i}}y$, $c=z$ we get that
\[\xi ^{p-1-i}(\xi ^{i} x+{\xi ^{i}}y)* z=
\xi ^{p-1-i}((\xi ^{i} x)*z)+\xi ^{p-1-i}(({\xi ^{i}}y)*z)+ \xi ^{p-1-i}C(i),\]
where $C(i)$ is a sum of some products (under the operation $*$) of elements $a=\xi ^{i}x$, $b=\xi ^{i}y$ and  $c=z$.

We would like to prove  that  $(x+y)\cdot z =x\cdot z+y\cdot z$ for $x,y\in pA$ and $z\in A$.
 By the above \[(x+y)\cdot z=\sum_{i=0}^{p-2}
 \xi ^{p-1-i}((\xi ^{i} x)*z)+\xi ^{p-1-i}(({\xi ^{i}}y)*z)+ \xi ^{p-1-i}C(i),\] and
  \[x\cdot z+y\cdot z=\sum_{i=0}^{p-2} \xi ^{p-1-i}((\xi ^{i} x)*z)+\xi ^{p-1-i}(({\xi ^{i}}y)*z).\]
 Consequently,  it is enough to prove that $\sum_{i=0}^{p-2}\xi ^{p-1-i}C(i)=0.$
  Let $V_{\xi ^{i}x, \xi ^{i}y, z}$ be a vector constructed as in Notation $1$, so entries of this vector  are products of elements $\xi ^{i}x$, $\xi ^{i}y$ and  $z$. We assume that $x,y\in pA$, $z\in A$, as mentioned before. By application of Corollary \ref{777}  followed  by  Lemma \ref{co},  every element from the set $E(\xi x, \xi y, z)$ can be written as a linear combination
   of elements from $E(x, y, z)$, with  integer coefficients which do not  depend on $x, y$ and $z$, provided that $x,y\in pA$. We can then  organize these coefficients into a matrix, which we will call $M=\{m_{i,j}\}$,  so that we obtain
 $MV_{x,y, z}=V_{\xi x, \xi y, z},$ for $x,y\in pA, z\in A$.

 Observe that  elements from $E(x,y, z)$ (and from $E(\xi x, \xi y, z)$) which are longer appear after elements which are shorter in our vectors $V_{x,y, z}$ and $V_{\xi x, \xi y, z}$. Notice that, by Corollary \ref{777} and  Lemma \ref{co}, $M$ is an upper triangular matrix.

 The first four elements in the  vector
 $V_{\xi x, \xi y, z}$ are $(\xi x)*((\xi y)*z)$, $((\xi x)*(\xi y))*z$, $(\xi y)*((\xi x)*z) $ and $ ((\xi y)*(\xi x))*z$ (arranged in some order).
  Suppose that $(\xi x)*((\xi y)*z)$ is the first entry in the vector $V_{\xi x, \xi y, z}$ (so $x*(y*z)$ is the first entry in the vector $V_{x, y, z}$).
 Recall that we have assumed that $x,y\in pA$, $z\in A$. By application of Corollary \ref{777} followed by Lemma \ref{co}  several times  $(\xi x)*((\xi y)*z)$ can be written as a sum of element  $\xi ^{2}(x*(y*z))$ and some  elements  from $E(x,y,z)$ of length  larger than $3$ (so these elements are products of  four or more elements from the set $\{x,y,z\}$). Therefore  the first diagonal entry in $M$ equals $\xi ^{2}$, so $m_{1, 1}=\xi ^{2}$.

 Note  that the following diagonal entries of $M$ are  $\xi ^{i}$ for some natural numbers $i$ such that  $1<i< p-1$. Observe that   $i<p-1$, because any product where $\xi ^{p-1}$ appears would have $p-1$ occurrences of elements from the set $\{x,y\}$; however, $x,y\in pA$, and so any such product would belong to $p^{p-1}A=0$, by Proposition \ref{b}, and for this reason  such elements were not included in the definition of $E(x,y,z)$ and $V_{x,y,z}$.
 Note also that $i>1$, since any element from $E(a,b, c)$ contains both $a$ and $b$.

As an another example, note that  $(\xi x)*((\xi x)*((\xi x)*z))$ can be written using Corollary \ref{777} and Lemma \ref{co} as $\xi ^{3}(x*(x*(x*z)))$ plus  elements  from $E(x,y,z)$ of length at least $5$.
Hence  $M$ is an upper triangular matrix with all diagonal entries of the form $\xi ^{i}$, where $1<i< p-1$.
Therefore all diagonal entries of the matrix  $\xi ^{(p-1)-1}M$ are $\xi ^{p-2+i}$ for some  $1<i<p-1$.

Note that $\xi ^{p-2+i}\equiv \xi ^{i-1}\mod p^{p}$, since $\xi ^{p-1}\equiv 1 \mod p^{p}$, so the diagonal entries are congruent to $\xi ^{i-1}$ modulo $p$, where $0<i-1<p-1$.
By Lemma  \ref{Engelxi},  diagonal entries of   the matrix $\xi ^{p-1-1}M$ are
not congruent to $1$ modulo $p$.

Note also that $M$ does not depend on $x, y$ and $z$, as we only used relations from Lemma \ref{co} to construct it.

Consequently, for $x,y\in pA$, $z\in A$ and  for every $i$, we have  $V_{\xi ^{i}x, \xi ^{i}y, z}=M^{i}V_{x,y, z}.$ 
Note that  $\xi ^{p-1}x=x$ and $\xi ^{p-1}y=y$, because  $\xi ^{p-1}$ is congruent to $1$ modulo $p^{n}$ and $p^{n}x=p^{n}y=0$, since the group $(A, +)$ has cardinality $p^{n}$ and $n<p-1$. Consequently,
$V_{x, y, z}=V_{\xi ^{p-1}x, \xi ^{p-1}y, z}=M^{p-1}V_{x, y, z}=(\xi ^{(p-1)-1}M)^{p-1}V_{x, y, z}$
for $x,y\in pA$, and $z \in A$.

Notice that there exists a vector $V$ with integer entries such that,
\[ C(i)=V^{T}V_{\xi ^{i}x, \xi ^{i}y, z}=V^{T}M^{i}V_{x, y, z},\]  for each $i$, where $V^{T}$ is the transpose of $V$.

Denote $M^{0}=I$, the identity matrix.  Recall that $\xi ^{p-1}\equiv 1 \mod p^{n}$. It follows that 
\[\sum_{i=0}^{p-2}\xi ^{(p-1)-i}C(i)=\sum_{i=0}^{p-2}(\xi ^{(p-1)-1})^{i}C(i)=V^{T}Q_{x,y,z}\] where 
$Q_{x,y,z}=\sum_{i=0}^{p-2}(\xi ^{(p-1)-1}M)^{i}V_{x,y,z}.$ It remains to prove that all entries of vector $Q_{x,y,z}$ are zero. Notice that 
\[(I-\xi ^{(p-1)-1}M)\sum_{i=0}^{p-2}(\xi ^{(p-1)-1}M)^{i}=I-(\xi ^{(p-1)-1}M)^{p-1},\]
 consequently, 
 \[(I-\xi ^{(p-1)-1}M)Q_{x,y,z}= (I-(\xi ^{(p-1)-1}M)^{p-1})V_{x,y,z}=0,\] where $I$ is the identity matrix of the same dimension as $M$.  
 Recall
that entries of the vector $Q_{x,y,z}$ are elements of $A$ and therefore their additive orders are
powers of $p$. Recall that the diagonal entries of the matrix $I-\xi ^{(p-1)-1}M$ are coprime
with $p$, and that this matrix is upper triangular. Therefore $Q_{x,y,z}=0$.

$ $

{\bf Part 2.} We will show that
 $(a\cdot b)\cdot c-a\cdot (b \cdot c)=(b\cdot a)\cdot c-b\cdot (a\cdot c),$
for every $a, b\in pA$, $c\in A$. We will use a proof similar to the proof of Theorem $6$ from \cite{passage2}.
   Assume that $x,y\in pA, z\in A$.
By Lemma \ref{co}  we get
\[(x+y)*z=x*z+y*z +x*(y*z)-(x*y)*z +d(x,y,z),\]
\[ (y+x)*z=x*z+y*z+y*(x*z)-(y*x)*z +d(y,x, z),\]
where $d(x, y, z)=E^{T}V_{x, y, z}$ for some vector $E$ with integer entries (and these entries do not depend on $x, y, z$)  and where $V_{x, y, z}$ is a vector which occured in part $1$ of our proof above. Observe that  $d(x, y, z)$ is a combination of elements with three or more occurrences of elements from the set $\{x, y\}$). 

 We will now use  lines 7-27 of the proof of Theorem $6$ in \cite{passage2}, the only difference being that in line $18$ we need to use part $1$ of our proof here (above) instead of Proposition $5$ from \cite {passage2}.
 We then get
 that to show that
 \[(x\cdot y)\cdot z-x\cdot (y \cdot z)=(y\cdot x)\cdot z-y\cdot (x\cdot z),\]
for every $x,y \in pA$, $z\in A$ it suffices to show that

 \[\sum_{i,j=0}^{p-2}\xi ^{p-1-i+p-1-j}d(\xi ^{i}x, \xi ^{j}y,z)=0,\] for all $x,y \in pA,  z\in A.$

{\em Proof that  $\sum_{i,j=0}^{p-2}\xi ^{p-1-i+p-1-j}d(\xi ^{i}x, \xi ^{j}y, z)=0$} for $x,y\in pA,$  and $z\in A$:

 Observe that $d(x,y,z)=w(x,y,z)+v(x,y,z)$, where $w(x,y,z)$ contains all the products of elements $x,y,z$ which appear as summands in  $d(x,y,z)$ and in which $x$ appears at least twice, and
$v(x,y,z)$ is a sum of products which are summands in $d(x,y,z)$ and in which $x$ appears only once (and consequently $y$ appears at least twice).
It is sufficient to show that $\sum_{i,j=0}^{p-2}\xi ^{p-1-i+p-1-j}w(\xi ^{i}x, \xi ^{j}y, z)=0$ and $\sum_{i,j=0}^{p-2}\xi ^{p-1-i+p-1-j}v(\xi ^{i}x, \xi ^{j}y, z)=0$.
 It suffices to show that $\sum_{i=0}^{p-2}\xi ^{p-1-i}w(\xi ^{i}x, y',z)=0$
and  $\sum_{j=0}^{p-2}\xi ^{p-1-j}v(x', \xi ^{j}y, z)=0$ for any $x, x',  y, y'\in pA, z\in A$ (note that $x', y'$ should not be confused with inverses of $x$ and $y$).

We will first show that \[\sum_{i=0}^{p-2}\xi ^{p-1-i}w(\xi ^{i}x, y',z)=0.\]

Note that there is a vector $W$ with integer entries such that $w(x,y',z)=W^{T}V_{x,y',z}'$, where $V_{x,y',z}'$ is a vector constructed similarly as in Notation $1$ but only including as entries  products from $E(x,y',z)$  in which $x$ appears at least twice (namely, for $x,y'\in pA$, $z\in A$, let ${\tilde E}(x,y',z)$ be the subset of  $E(x, y', z)$ consisting of these products from $E(x,y',z)$ in which $x$ appears at least twice,  and let
 $V_{x,y',z}'$  be a vector whose entries products from ${\tilde E}(x,y',z)$ arranged in such way that shorter products of elements are situated before longer products).
 It follows that
\[ w(\xi ^{i}x, y',z)=W^{T}V'_{\xi ^{i}x, y',z}=W^{T}M^{i}V'_{x, y',z}\] for each $i$, and for all $x, y'\in pA, z\in A$, where $W^{T}$ is the transpose of $W$, and $M$ is some upper triangular  matrix with integer entries such that
  $V'_{\xi x, y, z}=MV'_{x,y',z}$ and hence $V'_{\xi ^{i}x, y, z}=M^{i}V'_{x,y',z}$ for every $i$  (it can be shown that such a matrix $M$ exists as
in part $1$ of this proof, since $x, y'\in pA$).

  Arguing as in the first part of this proof, we see that all diagonal entries of the matrix $I-\xi ^{(p-1)-1}M$ are coprime to $p$, by Lemma \ref{Engelxi}. It follows because  the
 diagonal entries of $M$ are $\xi ^{i}$ where
 $i \geq 2$, because $a$ appears more than once in each product which is an entry in $V'_{x, y', z}$  (and $i<p-1$  since all the products of length  longer than $p-2$ are zero, by Proposition \ref{b}).

Now we calculate

\[\sum_{i=0}^{p-2}\xi ^{(p-1)-i}w(\xi ^{i}x, y',z) =\sum_{i=0}^{p-2}(\xi ^{(p-1)-1})^{i}w(\xi ^{i}x, y',z) =W^{T}\sum_{n=0}^{p-2}(\xi ^{(p-1)-1}M)^{i}V'_{x,y',z},\]
 where $\xi ^{0}=1$ and $M^{0}=I$, the identity matrix.

Note that
\[(I-\xi ^{(p-1)-1}M)\sum_{i=0}^{p-2} (\xi ^{(p-1)-1}M)^{i}=I-(\xi ^{(p-1)-1}M)^{p-1}.\]
 Reasoning as in the proof of Theorem $6$ in \cite{passage} we obtain that

\[\sum_{i=0}^{p-2}\xi ^{p-1-i}w(\xi ^{i}x, y',z)=0.\]  Consequently,
  \[0=W^{T}\sum_{i=0}^{p-2}\xi ^{p-1-i}w(\xi ^{i}x, y',z)= \sum_{i=0}^{p-2}\xi ^{p-1-i}w(\xi ^{i}x, y',z),\]
  as required.

 The proof that
 \[\sum_{j=0}^{p-2}\xi ^{p-1-j}v(x', \xi ^{j}y, z)=0\] for all $x', y\in pA, z\in A$  is very similar to the proof of Theorem $6$ from \cite{passage2}, this proof is also very similar to the proof that $\sum_{i=0}^{p-2}\xi ^{p-1-i}w(\xi ^{i}x, y',z)=0$, so it is omitted.

\end{proof}

\section{The pullback $\wp ^{-1}$}

Let $A$ be a brace of cardinality $p^{n}$. 
For $a\in pA$ let $\wp^{-1}(a)$ denote an element $x\in A$ such that $px=a$. Such an element may not be uniquely determined in $A$, but we can fix for every $a\in pA$ such an element $\wp^{-1}(a)\in A$.
Notice that $p(\wp^{-1}(a))=px=a$.

 %Let $A$  be a brace, where $p$ is a prime number such that $p>n+1$.

For $A$ as above, $ann(p^{2})$ consists of all elements of $A$ which have additive order $p^{j}$ for $j\leq 2$. By Lemma $17$ from \cite{passage}, $ann(p^{2})$ is an ideal in $A$ (and so is $ann(p)$), provided that $p>n+1$.  Recall that if $I$ is an ideal in the brace $A$ then the factor brace $A/I$ is well defined.
In our situation $I=ann(p^{2})$, the elements of the brace $A/ann(p^{2})$ are cosets $[a]_{ann(p^{2})}:= a + ann(p^2)$ where $a \in A$, which we will simply denote by $[a]$, so $[a] =[b]$ if and only if $p^{2}(a-b)=0$.

\begin{lemma}\label{3} Let $A$ be a brace and let $A/ann(p^{2})$ be defined as above. 
Let $\wp^{-1} : pA \rightarrow A$ be defined as above. Then, for $a, b\in pA$ we have
$[\wp^{-1}(a)] + [\wp^{-1}(b)] = [\wp^{-1}(a + b)].$ 
This implies that, for any integer $m$ we have
 $[m \wp^{-1}(a)] = [\wp^{-1}(ma)].$
\end{lemma}

\begin{proof} Note that $[\wp^{-1}(a)] + [\wp^{-1}(b)] = [\wp^{-1}(a + b)]$ 
       is equivalent to
   $\wp^{-1}(a)+ \wp^{-1}(b) - \wp^{-1}(a + b) \in ann(p^{2})$ which is equivalent to
   $p^{2}(\wp^{-1}(a)+ \wp^{-1}(b)-\wp^{-1}(a + b)) = 0.$ This in turn is equivalent to
   $pa+pb-p(a+b) = 0,$ which holds true (since $p(\wp^{-1}(a)) = a$ for every $a\in A$, by the definition of the function $\wp^{-1}$).
  \end{proof}

\section{The binary operation $\odot $}

Let $A$ be a brace of cardinality $p^{n}$ where $p > n+1$ is a prime number. In this section we introduce a binary operation
$\odot : A/ann(p^{2})\times A/ann(p^{2})\rightarrow A/ ann(p^{2})$.
\begin{lemma}\label{666}
Let $(A, +, \circ )$ be a brace of cardinality $p^{n}$, for a prime number $p>n+1$. Let $\wp^{-1}:pA\rightarrow A$ be defined as in the previous section. Let $a,b\in A$.  
Define
\[[a]\odot [b]=[\wp^{-1}((pa)*b)].\]
Then this is a well defined binary operation on $A/ann(p^{2})$.
\end{lemma}
\begin{proof} We need to  show that $\odot $ is a well defined operation, so that the result does not depend on the choice of coset representatives $x, y$ of cosets $[x], [y]$.

Let $x, x'\in A$ be elements satisfying $[x]=[x']$; then $p^{2}(x-x')=0$, which is equivalent to $p(x-x')\in ann(p)$.
To show that $[x]\odot [y]$ does not depend on the representative $x$
we need to show that $\wp^{-1}((px)*y)-\wp^{-1}((px')*y) \in ann(p^{2}).$
This is equivalent to $p((px)*y)-p((px')*y)=0.$
Applying Lemma \ref{citeEngel} for $a=p(x-x')$, $b=px'$ and $c=y$, we obtain
\[(px)*y=(p(x-x')+(px'))*y=(px')*y+e ,\]
where $e$ belongs to the ideal generated in  the brace $A$ by the element $p(x-x')$.
It follows that  $e\in ann(p)$ since $p(x-x')\in ann(p)$
and $ann(p)$ is an ideal in $A$ by Lemma $17$ of \cite{passage}.
We conclude that $pe=0$, hence \[p((px)*y)-p(( px')*y)=0,\] as required. Therefore $[x]\odot [y]$ does not depend on the choice of the representative for the coset $[x]$.

 Next we will show that $[x]\odot [y]$ does not depend on the choice of the representative for the coset $[y]$. We need to show that, if $y' \in A$ satisfies $[y] = [y']$ then 
\[\wp^{-1}((px)*y)-\wp^{-1}((px)*y')\in ann(p^{2}).\]
By the above, this is equivalent to $p((px)*y)-p((px)*y')=0.$
By the defining relations of any brace we have
$p((px)*y)-p(( px)*y')=(px)*(p(y-y')).$
 By Lemma \ref{citepassage}, $px=u_{1}^{\circ p}\circ u_{2}^{\circ p}\circ \cdots \circ u_{s}^{\circ p}$ for some $u_{1}, \ldots , u_{s}\in A$, and by Remark \ref{555} we get $pa=u^{\circ p}$ for some $u\in A$. 
 By the formula from Lemma \ref{citenote}  we get
 $(px)*(p(y-y'))=u^{\circ p}*(p(y-y'))=0,$ since $p^{2}(y-y')=0$ (since $A^{p}=0$ by \cite{rump}, page 166). 
 Therefore  the result does not depend on the choice of the representative for the coset $[y]$.
\end{proof}

\section{ Main results}
 Our main result is the following:

\begin{theorem}\label{dc}
 Let $(A, +, \circ )$ be a brace of cardinality $p^{n}$, where $p$ is a prime number such that  $p>n+1$. Let $\odot $ be defined as in the previous section, so
$[x]\odot [y]=[\wp^{-1}((px)*y)]$. Let $\xi =\gamma  ^{p^{p-1}}$ where $\gamma $ is a primitive root modulo $p^{p}$.
 Define a binary  operation $\bullet $ on $A/ann(p^{2})$ as follows:
\[[x]\bullet [y]=\sum_{i=0}^{p-2} \xi ^{p-1-i} [\xi ^{i}x]\odot [y],\]
for $x,y\in A$. Then $A/ann(p^{2})$ with the binary operations $+$ and $\bullet $ is a pre-Lie ring.
\end{theorem} 
\begin{proof} 
\noindent
{\bf Part 1}. We need to show that  $([a]+[b])\bullet [c]=[a]\bullet [c]+[b]\bullet [c]$ and
$[a]\bullet ([b]+[c])=[a]\bullet [b]+[a]\bullet [c]$ for every $a,b,c\in A$.
Observe that the formula $[x]\bullet ([y]+[w])=[x]\bullet [y]+[x]\bullet [w]$ follows immediately from the fact that in every brace $x*(y+w)=x*y+x*w$.
We will show now that \[([x]+[y])\bullet [w]=[x]\bullet [w]+[y]\bullet [w].\]
 Notice that $[x]\bullet [y]=[\wp ^{-1}(px\cdot y)]$, where operation $\cdot $ is as in Proposition \ref{12345}. 
 By Proposition \ref{12345} and Lemma \ref{3} 
\[([x]+[y])\bullet [z]=[\wp ^{-1}((px+py)\cdot z)]=[\wp ^{-1}(((px)\cdot z)+((py)\cdot z))]=\]
\[=[\wp ^{-1}((px)\cdot z)]+[\wp ^{-1}((py)\cdot z)]=[x]\bullet [z]+[y]\bullet [z].\]
\medskip

{\bf Part 2.} We need to show that \[([x]\bullet [y])\bullet [z]-[x]\bullet ([y]\bullet [z])=([y]\bullet [x])\bullet [z]-[y]\bullet ([x]\bullet [z]),\]
for $x, y,z\in A$. Recall that $[x]\bullet [y]=[\wp^{-1}((px)\cdot y)].$ 
  Consequently,  \[([x]\bullet [y])\bullet [z]=[\wp^{-1}(px\cdot y)]\bullet[z]=[\wp^{-1} (((px)\cdot y)\cdot z)],\] and
\[[x]\bullet ([y]\bullet [z])=[x]\bullet [\wp^{-1}((py)\cdot z) ]=[\wp^{-1}((px)\cdot \wp^{-1}((py)\cdot z))].\]

Note that   \[p^{2}(\wp^{-1}((px)\cdot \wp^{-1}((py)\cdot z)))=p((px)\cdot \wp^{-1}((py)\cdot z))=(px)\cdot ((py)\cdot z),\]

On the other hand we have
$p^{2}(\wp^{-1} (((px)\cdot y)\cdot z))=p(((px)\cdot y)\cdot z).$
Note that $(px)\cdot  y\in pA$, since, by Lemma \ref{citepassage}, $pA$ is an ideal in $A$.

Let $a\in pA$, $b\in A$,
then $(na)\cdot b=n(a\cdot b)$, since \[(a+a+\cdots +a)\cdot b=a\cdot b+a\cdot b+\cdots +a\cdot b\]  by Proposition \ref{12345}.
Applying this for $n=p$, $a=(px)\cdot y$ and  $b=z$, we get that \[p(((px)\cdot y)\cdot z)=(p((px)\cdot y))\cdot z=((px)\cdot (py))\cdot z.\]

Therefore, \[([x]\bullet[y])\bullet [z]-[x]\bullet  ([y]\bullet [z])=([y]\bullet [x])\bullet [z]-[y]\bullet ([x]\bullet [z]),\] written using  this notation, and multiplied by $p^{2}$ gives an equivalent condition

\[((px)\cdot (py))\cdot z- (px)\cdot ((py)\cdot z)=((py)\cdot (px))\cdot z -(py)\cdot ((px)\cdot z).\]
  This condition holds by Proposition \ref{12345} applied for $a=px$, $b=py$, $c=z$. This concludes the proof.
\end{proof}
{\em Connection with the Group of Flows}. The group of flows of a pre-Lie algebra was invented in \cite{AG} and  in \cite{Rump}   Rump  suggested to use this construction to obtain a passage from pre-Lie algebras to $\mathbb F_{p}-$ braces. Notice that if $A$  is a strongly nilpotent brace of nilpotency index $k<n$ and cardinality $p^{n}$ with $n+1<p$ for a prime number $p$, then we can define a pre-Lie ring
$(A, +, \cdot)$ associated to the brace $A$ as in Proposition $5$ and Theorem $6$ in \cite{passage} (a small variation of this pre-Lie ring is obtained by  reversing the construction of the group of flows from \cite{Rump}).  Recall the formula from \cite{passage}:  $x\cdot y= \sum_{i=0}^{p-2} \xi ^{p-1-i} (\xi ^{i}x)*y,$
for $x,y\in A$.
Observe that  $[a]\bullet [b]=[\wp^{-1} ( \sum_{i=0}^{p-2} \xi ^{p-1-i} (\xi ^{i}pa)*b))]=
[\wp^{-1}( (pa)\cdot b )]=[\wp^{-1}(p(a\cdot b))]=[a\cdot b].$

It was shown in \cite{passage} that by applying the group of flows construction to the pre-Lie ring $(A, +, \cdot _{1})$ yields the original brace $(A, +, \circ )$, where $a\cdot _{1} b=-(1+p+ \cdots +p^{n})a\cdot b$, as mentioned in Notation $2$ in \cite{passage}. Therefore, by applying the group of flows construction to the pre-Lie ring $(A/ann(p^{2}), +, \bullet _{1} )$,
where $a\bullet _{1} b=-(1+p+\cdots +p^{n})a\bullet b$ for $a,b\in A$,
yields the brace $(A/ann(p^{2}), +, \circ )$.

$ $

The forthcoming paper \cite{newSmok} investigates how to recover  the brace $(A/ann (p^{2}), +, \circ )$ from the obtained pre-Lie ring $(A/ann (p^{2}), +, \bullet )$ in the general case when $A$ need not be right nilpotent (see also \cite{cl}). A correspondence between braces  $A/ann(p^{2})$ and a subset of left nilpotent pre-Lie rings with the same additive group is also investigated.

\vspace{5mm}

{\bf Acknowledgements.} We are grateful to the referee for many helpful suggestions which improved the original manuscript. In particular we are grateful for improving and clarifying  the second page of our introduction and for Remark \ref{555}, and for suggestions as to how to use this remark to simplify the  proof of Lemma \ref{666}.
  
The first author acknowledges support by the ISF grant 700/21, the BSF grant 2020/037 and the Vinik Chair
of mathematics which he holds.
The second author acknowledges support from the EPSRC programme grant EP/R034826/1 and from the EPSRC research grant EP/V008129/1.
Both authors are grateful to Leandro Vendramin and Wolfgang Rump for useful comments.

\pagebreak

\def\F{{\mathbb F}}

\begin{center}
{\large\bf Groups, Lie algebras and braces}

\end{center}

\begin{center}
\author{Aner Shalev, Agata Smoktunowicz}
\end{center}
\date{}

\maketitle
\begin{abstract} We use powerful Lie rings associated with finite $p$-groups in the study of brace automorphisms with few fixed points.  As an application  we  bound the number of elements which commute with a given element in a brace, as well as the number of elements which multiplied from left by a given element give zero.  
 We also study various Lie rings associated to powerful groups and braces whose adjoint groups are powerful, and show that the obtained Lie and pre-Lie rings are also powerful. 

We also show that  braces whose adjoint groups are powerful and  powerful left nilpotent pre-Lie rings  are in one-to-one correspondence and that they are  right nilpotent under some  cardinality assumptions.
%.
%We  also show that braces whose adjoint groups are powerful are strongly nilpotent.
 %Similarly, we show, that left nilpotent pre-Lie rings whose Lie rings are powerful are right nilpotent, and hence  nilpotent. 
\end{abstract}

\section{Introduction} 

Braces were introduced in 2007 by Wolfgang Rump \cite{rump}; they are a generalisation of Jacobson radical rings with the two-sided braces being exactly the Jacobson radical rings.
 
One of the main motivations for investigating braces is their connection with set-theoretic solutions of the Yang-Baxter equation \cite{rump}.
In this paper we investigate some group theory results and their applications in the theory of braces.  Notice also that braces can be used in the investigation of groups. 
 For example, it is known that all groups of nilpotency class two \cite{Leandro} are adjoint groups of braces,
 hence nilpotent groups of class two can be investigated using brace theory methods.

 %One of the main motivations for investigating braces is their connections with
% set-theoretic solutions of the Yang-Baxter equation \cite{rump, cjo} (see also \cite{doikou, doikou2, pent} for  
% some connections with mathematical physics); another motivation is the
% connections of braces with homological group theory, since braces are exactly groups with bijective 1-cocycles.
% Another motivation to investigate braces is the
% connection of braces with homological group theory, since braces are exactly groups with bijective 1-cocycles.

%*******************************************
%The theory of braces is also connected to algebraic number theory and its generalisations through the concept of Hopf-Galois extensions of abelian type \cite{Bachiller}. 
 %It was shown by Gateva-Ivanova in \cite{gateva} that braces are in one-to-one correspondence with involutive braided groups, a structure which has been used to investigate involutive set-%theoretic solutions of the Yang-Baxter equation since 1999.

  There also exist connections between braces and homological group theory, since braces are exactly groups with bijective 1-cocycles.  See \cite{Bachiller, doikou, doikou2, pent, gateva} for  
 some connections with mathematical physics, braided groups and Hopf-Galois extensions. 

 %On the other hand, connections between braces and pre-Lie algebras make it possible to find connections between braces and %symmetric brace algebras \cite{bracealgebras}, previously %unrelated concepts.

On the other hand, connections between braces and pre-Lie algebras make it possible to find connections between braces and symmetric brace algebras \cite{bracealgebras}, previously unrelated concepts.

The  connections between braces and pre-Lie algebras go back to the fundamental paper by Rump \cite{Rump}, where he
% passes from finite pre-Lie algebras to finite braces. He
 used the exponential function on a pre-Lie algebra to define a multiplication operation and showed that the obtained structure, together with the same addition as in the pre-Lie algebra, is a brace.
 He also described a passage from $\mathbb R$-braces to pre-Lie rings. For a general case of smooth Lie groups and their associated braces the passage to pre-Lie rings was described recently in \cite{BaiGuo}. The passage from finite braces to finite pre-Lie rings was investigated in \cite{Rump, Lazard, passage, asas}. 
 Moreover,  up to elements of order $p^{4}$ braces of cardinality $p^{n}$ for a prime number $p>n+1$  
 are exactly groups of flows of pre-Lie rings, with the same addition as in the original pre-Lie ring \cite{Newsmok}. 
This makes it possible for braces of this cardinality to be investigated by pre-Lie and
 Lie algebra researchers who do not have experience with braces.

In Section \ref{9} we investigate  when various Lie and pre-Lie rings (constructed in \cite{Shalev, asas}) associated to braces and to powerful groups are powerful. 
In Section \ref{9} we also derive an application to powerful $p$-groups of a given coclass.
 In Section \ref{11}  we show that  braces whose adjoint groups are powerful and  powerful left nilpotent pre-Lie rings  are in one-to-one correspondence and that they are  left and right nilpotent under some assumptions on cardinality.

In Section \ref{10} we consider applications of fixed points theorems applied to the $\lambda _{a}$ maps of a brace (recall that  $\lambda _{a}(b)=a*b+b=a\circ b-a$) which, when the brace has a small number of elements $b$ such that $a*b=0$, would give some information about the structure of the multiplicative group of the brace. Some results about braces whose adjoint groups are powerful are also obtained, as well as a  result about pre-Lie rings whose Lie-rings are powerful (Sections \ref{11}, \ref{12}). In Section \ref{13}, which is connected to Section \ref{9}, we turn again to finite $p$-groups and their graded Lie rings,  and pose some open questions.
Recall that a set $A$ with binary operations $+$ and $* $ is a {\em  left brace} if $(A, +)$ is an abelian group and the following version of distributivity combined with associativity holds.
  \[(a+b+a*b)* c=a* c+b* c+a* (b* c), \space  a* (b+c)=a* b+a* c,\]
for all $a, b, c\in A$; moreover  $(A, \circ )$ is a group, where we define $a\circ b=a+b+a* b$.
In what follows we will use the definition in terms of the operation `$\circ $' presented in \cite{cjo} (see \cite{rump}
for the original definition): a set $A$ with binary operations of addition $+$ and multiplication $\circ $ is a brace if $(A, +)$ is an abelian group, $(A, \circ )$ is a group and for every $a,b,c\in A$
\[a\circ (b+c)+a=a\circ b+a\circ c.\]

Let $(A, +, \circ)$ be a brace, then the group $(A, \circ)$ is called the adjoint group of brace $A$; it is also called the multiplicative group of brace $A$.  The papers \cite{E, W} contain the seminal background results related to involutive set-theoretic solutions.

\section{Lie rings and groups}\label{9}

%At the Conference ``The algebra of the Yang-Baxter equation'' in  Bedlewo on July 11-15, 2022, Jan Okni{\' n}ski asked if the %restriction on the cardinality (namely, the assumption that $p>n+1$) is necessary for any one-to-one correspondence between %braces and pre-Lie rings.
%In this section we consider a related question about a correspondence between Lie rings and groups, and we also deal with non-%associative algebras, including pre 
%Lie-rings.

Classical Lie methods in group theory attach to any nilpotent group $G$ a graded Lie ring
\[
L(G) = \oplus_{i=1, \ldots , c} \; {G_i/G_{i+1}},
\]
where $c$ is the nilpotency class of $G$ and
$\{ G_i \}_{i=1}^{c+1}$ is the lower central series. Set $L_i = L_i(G) = G_i/G_{i+1}$. Each $L_i$ is an abelian group, written additively, and so is $L$.
The elements of $L_i$ are called homogeneous.
In what follows it will be convenient to define $L_i = 0$ for all $i > c$.

We denote group commutators by
\[
(x,y):= x^{-1}y^{-1}xy,
\]
where $x,y \in G$.

A Lie product on $L$ is first defined on homogeneous elements by
\[
[xG_{i+1}, yG_{j+1}] = (x,y) G_{i+j+1},
\]
where $G_{i+j+1} = 1$ if $i+j > c$.

It is easy to see that this is well-defined
(i.e. the product is independent of the representatives $x \in G_i, y \in G_j$), and that $[L_i, L_j] \subseteq L_{i+j}$.

The Lie product of homogeneous elements is extended to arbitrary elements using 
its bi-additivity. This gives $L(G)$ the structure of a nilpotent graded Lie ring of class $c$, 
and if $G$ is finite we also have $|L(G)|=|G|$. Note also that $L(G)$ is generated by $L_1(G)$ as a Lie ring.

In the important special case where $G$ is a finite $p$-group of exponent $p$, $L(G)$ is a Lie algebra
 over the field of size $p$ (and the same holds under the weaker assumption that each $G_i/G_{i+1}$ has exponent dividing $p$).

 It is known that the above construction of the Lie-ring from a nilpotent group is not reversible in general, even for nilpotent groups \cite{Khukhro}.

 In the context of braces, graded structures associated to braces were considered in
\cite{Iyudu} and were shown to be pre-Lie algebras. Notice that by using examples of strongly nilpotent braces constructed in \cite{Lazard,  dsk}  it can be shown that this construction is not reversible, even for strongly nilpotent braces. 

There are several other Lie and pre-Lie rings associated to groups and braces. 

In this section we obtain results on powerful groups and Lie rings,
 related to these Lie and pe-Lie ring constructions from  groups and braces  
 and apply them in the context of the coclass theory.

We need some notation. For a subgroup $H$ of a nilpotent group $G$ and $i \ge 1$ define
\[
L_i(G,H) = (H \cap G_i)G_{i+1}/G_{i+1} \cong (H \cap G_i)/(H \cap G_{i+1}),
\]
and note that $[L_i(G,H), L_j(G,H)] \subseteq L_{i+j}(G,H)$ for all $i,j \ge 1$.

Set $L(G,H) = \oplus_{i=1}^c L_i(G,H)$. Then $L(G,H)$ is a Lie-subring of $L(G)$ which we associate with the subgroup $H$ of $G$. Clearly, for subgroups
$H \le K \le G$ we have $L(G,H) \le L(G,K)$.

\begin{lemma}
With the above notation we have $L(G, G^p) \subseteq p L(G)$.
\end{lemma}

\begin{proof}
Recall that $L(G) = \oplus_{i=1}^c L_i(G)$ is the graded Lie ring associated with $G$, where $c$ is the nilpotency class of $G$.
We have $G^p = \langle g^p: g \in G \rangle$, the (normal) subgroup generated by the $p$th powers of the elements of $G$.
Since every element of $G^p$ is a product of $p$th powers, which are multiples of $p$ in $L(G)$, we obtain $L(G,G^p) \subseteq p L(G)$.

\end{proof}

Powerful $p$-groups, introduced by Lazard in his seminal work on $p$-adic analytic groups (see the book \cite{DDMS} by Dixon, du Sautoy, Mann and Segal), were studied further by Lubotzky and Mann \cite{LM1, LM2}.

We need some notation. Let $G$ be a finite group. For a positive integer $k$ define $G^k = \langle g^k : g \in G \rangle$, the subgroup generated by all $k$th powers in $G$. Clearly, $G^k \lhd G$.

We say that a finite $p$-group $G$ is {\em powerful} if $p>2$ and $G/G^p$ is abelian (namely, if $G' \le G^p$), or if $p=2$ and
$G/G^4$ is abelian (namely, if $G' \le G^4$).

Let $L$ be a finite Lie ring whose size is a power of $p$. Suppose $p > 2$.
We say that $L$ is powerful if $L/pL$ is an abelian Lie ring, namely if $L^2 \le pL$. In the case $p=2$ we say that $L$ is powerful if $L/4L$ is an abelian Lie ring, namely if $L^2 \le 4L$.

\begin{corollary}\label{Cor22} Let $G$ be a finite $p$-group and let $L(G)$ be its graded Lie ring. Suppose $G$ is powerful. Then so is $L(G)$.
\end{corollary}

\begin{proof}
Note that, for any nilpotent group $G$ we have $L(G)^2 \le L(G,G')$.

Indeed, $L(G)^2 = \oplus_{i=2}^c L_i(G)$ and for $i \ge 2$ we have
\[
L_i(G,G') \cong (G' \cap G_i)/(G' \cap G_{i+1}) = G_i/G_{i+1}.
\]

Since $G$ is powerful we have $G' \le G^p$. It follows that
\[
L(G)^2 \le L(G,G') \le L(G,G^p) \le p L(G),
\]
so $L(G)$ is a powerful Lie ring.

\end{proof}

\begin{lemma} Let $L$ be a powerful Lie ring of size $p^n$. Then $L^{n+1}=0$.
\end{lemma}

\begin{proof}
We claim that $L^{i+1} \le p^i L$ for all $i \ge 1$. The proof is by induction on $i$. The base case $i=1$ is equivalent to the assumption that $L$ is powerful (i.e. $L^2 \le pL$).

For the inductive step, suppose $L^{i+1} \le p^i L$. Then
$$L^{(i+1)+1} = [L, L^{i+1}] \le [L, p^i L] = p^i [L,L] = p^i L^2 \le p^i \cdot p L = p^{i+1} L,$$
as required.

\end{proof}

\begin{lemma}\label{24} Let $L$ be a powerful Lie ring whose size is a $p$th power. Then $L^{i+1} \le p^i L$ for all $i \ge 1$.
\end{lemma}

\begin{proof}
The proof is by induction on $i$. The base case $i=1$ is equivalent to the assumption that $L$ is powerful (i.e. $L^2 \le p L$). For the inductive step, suppose $L^{i+1} \le p^i L$. Then $L^{(i+1)+1} = [L, L^{i+1}] \le [L, p^i L] = p^i [L,L] = p^i L^2 \le p^i \cdot p L = p^{i+1} L,$ as required.

\end{proof}

We now present applications to the coclass theory.
Recall that the coclass of a group of order $p^n$
and nilpotency class $c$ is defined by $cc(G):= n-c$.
For background on the coclass theory, see the book
\cite{LeMc} by Leedham-Green and McKay and the references therein. Recall that the proof of the coclass conjectures applied the theory of $p$-adic analytic pro-$p$ groups and of finite powerful $p$-groups (see \cite{Sh-Z} and \cite{Sh2}).

\begin{proposition} Let $G$ be a group of order $p^n$ and coclass $b$. Suppose $G$ is powerful. Then $n \le 2b+1$.
\end{proposition}

\begin{proof}
Denote the nilpotency class of $G$ by $c$ and let $L = L(G)$. Then
$L = \oplus_{i=1}^c L_i$ has size $p^n$ and class $c = n-b$. Since $G$ is powerful, so is $L$ by Corollary \ref{Cor22}

Applying Lemma \ref{24}, we see that $L^{i+1} \subseteq p^i L$ for all $i \ge 1$. In particular $L^c \subseteq p^{c-1} L$. Since $L^c \ne 0$ it follows that $p^{c-1} L \ne 0$. Therefore there exists $j \in [1,c]$ such that $p^{c-1}L_j \ne 0$.

This shows that the exponent of $L_i = G_i/G_{i+1}$ as an additive group is at least $p^c$. In particular this yields
$$|G_j/G_{j+1}| \ge p^c.$$
For $i \in [1,c]$ set $|L_i|=p^{d_i}$. Then $d_j \ge c$ and $\sum_{i=1}^c d_i = n$. Therefore $$\sum_{i=1}^c (d_i-1) = n-c =b.$$
Since $d_i-1 \ge 0$ for all $i \in [1,c]$ we have
$$n-b-1 \le c-1 \le d_j-1 \le \sum_{i=1}^c (d_i-1) = n-c =b.$$
We conclude that $n \le 2b+1$,
as required.

\end{proof}

Note also that in \cite{TW} the class of powerfully nilpotent groups, a special subclass of powerful groups, was investigated.

\begin{lemma}\label{pw}
Let $G$ be a finite powerful  $p$-group, and let $x\in G^{p}$ and $y\in G$. Then
\[ (x,y) \in G^{p^{2}}.\]
\end{lemma}
\begin{proof}
Recall that $(x,y) = x^{-1}y^{-1}xy$. Let $G$ be a finite powerful $p$-group and let $N \lhd G$. We say that $N$ is powerfully embedded in $G$ if $p>2$ and $(N,G) \le N^p$ or $p=2$ and $(N,G) \le G^4$. Note that, since $N' = (N,N) \le (N,G)$, this implies that $N$ is powerful.

By \cite[Lemma 3.1]{Shalev}, if $G$ is a powerful $p$-group and $M,N \lhd G$ are powerfully embedded in $G$, then $$(N^{p^i},M^{p^j}) = (N, M)^{p^{i+j}}$$ for all $i,j \ge 0$.
It is well-known that, for $G$ as above, the power subgroups $G^{p^i}$ ($i \ge 0$), and the subgroups $\gamma_k = \gamma_k(G)$ ($k \ge 1$) forming the lower central series of $G$, are all powerfully embedded in $G$ (see \cite{LM1} or \cite{DDMS}).

Applying this to $N=G^p$ and $M=G$ we obtain
\[
(x,y) \in (G^p,G) = (G,G)^p = (G')^p \le (G^p)^p = G^{p^2},
\]
as required.
\end{proof}

Another consequence of \cite[Lemma 3.1]{Shalev} is the following.

\begin{lemma}\label{i,j}
Let $G$ be a finite powerful $p$-group and let $\{\gamma_k\}_{k \ge 1}$ be its lower central series. Then we have ($G$ replaced by $\gamma$ below):
\[
(\gamma_k^{p^i},\gamma_l^{p^j}) = \gamma_{k+l}^{p^{i+j}}
\]
for all $i,j \ge 0$ and $k,l \ge 1$.
\end{lemma}
The deduction of the above lemma from \cite[Lemma 3.1]{Shalev} is immediate.

Next, we consider the finite analogs of Lazard's Lie ring of $p$-adic analytic pro-$p$ groups (cf. \cite{Shalev} for details). A finite powerful $p$-group $G$ of exponent $p^e$ and $d$ generators is called {\em uniform} if $|G^{p^i}:G^{p^{i+1}}| = p^d$ for all $0 \le i < e$. Given such a group, let $0 \le i \le e/3$. Consider the
map $q: G^{p^i}/G^{p^{2i}} \to G^{p^{2i}}/G^{p^{3i}}$ defined by
\[
q(xG^{p^{2i}}) = x^{p^i}G^{p^{3i}}.
\]
Then $q$ is a well-defined group isomorphism, and $$G^{p^i}/G^{p^{2i}} \cong G^{p^{2i}}/G^{p^{3i}} \cong C_{p^i} \times \cdots \times C_{p^i},$$ the direct product of $d$ cyclic groups of order $p^i$.

Then the abelian group $L = G^{p^i}/G^{p^{2i}}$ becomes a Lie ring by defining for $x,y \in G^{p^i}$,
\[
x G^{p^i} + y G^{p^i} := xy G^{p^i},
\]
and
\[
[x G^{p^i} , y G^{p^i}] := q^{-1}((x,y) G^{p^{2i}}),
\]
where $(x,y) = x^{-1}y^{-1}xy$, the group commutator. We denote this Lie ring by $M = M_i(G)$.

\begin{proposition}
$M$ is a powerful Lie ring.
\end{proposition}

\begin{proof}
Formula (1) in \cite[p. 277]{Shalev} shows that
\[
(G^{p^i})' = (G')^{p^{2i}} \le (G^p)^{p^{2i}} = G^{p^{2i+1}}.
\]
This implies
\[
q^{-1}((G^{p^i})'G^{p^{3i}}/G^{p^{3i}}) \le q^{-1}(G^{p^{2i+1}}G^{p^{3i}}/G^{3i}) = q^{-1}(G^{p^{2i+1}}/G^{3i}) = G^{p^{i+1}}/G^{p^{2i}}.
\]
This shows that $M^2 \le p M$, hence $M$ is a powerful Lie ring.
\end{proof}
 We will  say that a pre-Lie ring $(A, +, \cdot )$ is powerful if the associated Lie ring $(A, +, [.,.])$ with $[a,b]=a\cdot b-b\cdot a$ is powerful. 
In \cite{TW}, on page $797$ (lines $7-9$), it was shown that Lie rings associated with powerful groups by using the Lazard's correspondence are powerful Lie rings.
Notice that this implies the following:
\begin{proposition}  Let $p$ be a prime number and let $n$ be a natural number such that $n+1<p$. 
 Let $(A, +, \cdot)$ be a finite left nilpotent pre-Lie ring of cardinality $p^{n}$, and let  $(A, \circ)$ be its group of flows (notice that $(A, \circ)$ is  well defined, even if $A$ does not have a filtration by  \cite{passage}).
 Suppose that the group $(A, \circ )$ is powerful, then the pre-Lie ring $(A, +, \cdot )$ is powerful. 
\end{proposition}
\begin{proof} Notice that this follows from the known fact that the group of flows of the pre-Lie ring $(A, +, \cdot)$ is isomorphic  to the group obtained by applying the Lazard's correspondence to the Lie ring $(A, +, [. , .])$. Therefore, the Lie ring associated to the group $(A, \circ )$ by using the Lazard correspondence is isomorphic to the Lie ring 
$(A, +, [.,.])$. 
For more details see \cite{AG} (or see the  Remark at the end of section $4$ in  \cite{Lazard}). 
\end{proof} 
Combining this result with Theorem $6$ from \cite{passage} we get the following:
\begin{proposition}\label{345}
 Let $p>2$ be a prime number and let $n$ be a natural number such that $n+1<p$. 
 Let $(A, +, \circ )$ be a brace of cardinality $p^{n}$. Suppose that $A$ is a strongly nilpotent brace of nilpotency index 
 $k<p$.
 Suppose that the group $(A, \circ )$ is powerful. Then the pre-Lie ring $(A, +, \cdot)$ be defined as in Theorem $6$ in \cite{passage} is powerful. 
 Recall that the addition in  $(A, +, \cdot)$ is the same as the addition in the brace $A$, and the multiplication  
$\cdot$ is defined as 
\[a\cdot b= \sum_{i=0}^{p-2} \xi ^{p-1-i} (\xi ^{i}pa)*b)),\]
 where $\xi=\gamma ^{p^{p-1}}$ where $\gamma $ is a primitive root modulo $p^{p}$.
\end{proposition}

We show that a similar result holds for construction from Theorem $11$ \cite{asas}:

\begin{proposition}\label{describingconstruction}
Let $(A, +, \circ )$ be a brace of cardinality $p^{n}$, where $p$ is a prime number such that  $p>n+1$. Suppose that the group $(A, \circ )$ is powerful. Let $ann (p^{2})$ be defined as before, so $ann(p^{2})=\{a\in A: p^{2}a=0\}$. Then the pre Lie-ring $(A/ann(p^{2}), +, \bullet)$ defined as in Theorem $11$ in \cite{asas} is powerful.
\end{proposition} 
\begin{proof} We use the notation of Theorem $11$ \cite{asas}.
 Recall that
\[[x]\bullet [y]=[\wp^{-1} ( \sum_{i=0}^{p-2} \xi ^{p-1-i} (\xi ^{i}px)*y))].\]
for $x,y\in A$.
We need to show that
\[[[x],[y]]=[x]\bullet [y]-[y]\bullet [x]\in [pA].\]
Since $(A, \circ )$ is a powerful group, we can apply Lemma \ref{pw} to the group $G=(A, \circ)$, then  by using our previous notation in this paper we have $G^{p}=A^{\circ p}$ and $G^{p^{2}}=A^{\circ p^{2}}$. By Theorem \ref{pw},
if $x\in A^{\circ p}$ and $y\in A$ then
\[ x^{-1} \circ y^{-1}\circ x \circ y\in A^{\circ p^{2}}.\]
Therefore $x\circ y=(y\circ x)\circ q(x,y)=y\circ x+q(x,y)+(y\circ x)*q(x,y)$ for some $q\in A^{\circ p^{2}}$.
By Lemma $15$ \cite{passage} we have
$A^{\circ p}=pA$ and $A^{\circ p^{2}}=p^{2}A$.
Therefore, $q(x,y)\in p^{2}A$, and by the defining relations of a brace we get $(y\circ x)*q(x,y)\in p^{2}A$.
It follows that
$x\circ y-y\circ x\in p^{2}A,$ and hence
\[x*y-y*x\in p^{2}A,\] for $x\in A^{\circ p}=pA$ and $y\in A$.

We denote $[pA]=\{p[a]:a\in A\}$. Note that $[pa]=[a+\cdots +a]=[a]+\cdots +[a]=p[a]$.
We can now continue our proof. Let $a,b\in A$ and apply the above for $x=\xi ^{i} pa$, $y=b$.
We obtain:
\[[a]\bullet [b]=[\wp^{-1} ( \sum_{i=0}^{p-2} \xi ^{p-1-i} (\xi ^{i}pa)*b)]=\]
\[=[\wp^{-1} ( \sum_{i=0}^{p-2} \xi ^{p-1-i} b*(\xi ^{i}pa) + t)],\] for some $t\in p^{2}A$. Notice that $[\wp^{-1}(t)]\in [\wp^{-1}(p^{2}A)]\subseteq [pA]$.
Moreover  \[[\wp^{-1} ( \sum_{i=0}^{p-2} \xi ^{p-1-i} b*(\xi ^{i}pa)]=[\sum_{i=0}^{p-2} \xi ^{p-1}( b*a)]=(p-1)[b*a].\]
Therefore, \[[a]\bullet [b]-(p-1)[b]*[a]\in [pA].\]
Similarly, \[[b]\bullet [a]-(p-1)[a]*[b]\in [pA].\]

Note that since $A$ is powerful we get $a*b-b*a\in pA$. Consequently, \[[[x],[y]]=[x]\bullet [y]-[y]\bullet [x]\in [pA].\]
\end{proof}

\section{Braces admitting automorphisms with few fixed points}\label{10}

Lie methods were applied by Higman, Kostrikin-Kreknin, Shalev, Khukhro, Medvedev and others to study nilpotent groups admitting an automorphism with few fixed points. More details on this and on Lie methods in group theory can be found in \cite[Chapter 1]{DSS}.

A group automorphism is called {\em regular} if it has only one fixed point (namely the identity element).
It is easy to see that a finite group with a regular automorphism of order $2$ is abelian. What if we replace $2$ by an arbitrary prime $p$?

Higman showed in \cite{H1} that a nilpotent group with a regular automorphism of prime order $p$ has nilpotency class bounded above by $h(p)$
for a suitable function $h$ (now called the Higman function).
Higman's method was to pass to the graded Lie ring $L = L(G)$ and to prove the analogous result for $L$. He showed that
$h(p) \ge (p^2-1)/2$ for all odd primes $p$, but did not obtain
any explicit upper bound on $h(p)$.

Such an upper bound was subsequently provided by Kostrikin and Kreknin, showing that, for all $p > 2$,
\[
h(p) \le {{(p-1)^{2^{p-1}-1}} \over {p-2}}.
\]
As for automorphisms of order $p$ with few fixed points, a result of Medvedev \cite[Theorem 1]{Medvedev} shows that
a group of order $p^n$ admitting an automorphism of order $p$ with $p^m$ fixed points has a normal subgroup $N$ such that $|G:N| \le g(p^m)$ for a suitable function $g$, and the nilpotency class of $N$ is at most $h(p)$.
As an immediate consequence we obtain:
\begin{corollary}\label{121}
Let $A$ be a skew brace of size $p^n$, whose additive group  admits an automorphism of order $p$ with $p^m$ fixed points. Then the additive group of $A$ has a normal subgroup $N$ such that $|A:N| \le g(p^m)$ for a suitable function $g$, and the nilpotency class of $N$ is at most $h(p)\le {{(p-1)^{2^{p-1}-1}} \over {p-2}}$.
\end{corollary}

While the above results deal with automorphisms of prime order,  automorphisms of prime power order were subsequently studied.

In \cite{Shalev} finite $p$-groups admitting an automorphism of order $p^k$ with $p^m$ fixed points were studied. It was shown there that the derived length of such groups $G$ is bounded above by $f(p^k,p^m)$ for an explicit function $f$.

In particular, applying this for inner automorphisms $g \mapsto x^{-1}g x$ induced by a fixed element $x \in G$, we conclude that the derived length of a finite $p$-group $G$ is bounded above by an explicit (increasing) function of $|C_G(x)|$, the size of the centralizer of $x$ in $G$.

This result is somewhat counter-intuitive, implying that a finite $p$-group which has an element of small centralizer has small derived length; it extends a classical result of Alperin from 1962 in the case where the conjugating element $x$ has order $p$ \cite{Al}.  Notice that a fruitful discussion about centralizers of elements  in braces and skew braces appears in \cite{Facchini}.

A subsequent result of Medvedev \cite[Theorem 1.2]{Medvedev2} shows that
a group of order $p^n$ admitting an automorphism of order $p^k$ with $p^m$ fixed points has a normal subgroup $S$ such that $|G:S| \le g(p^m,p^k)$ for a suitable function $g$, and the derived length of $S$ is at most $f(p^k)$ for a suitable function $f$.

Recall also that it was shown by Rump that adjoint groups of finite braces are soluble (see line 14 on page 147 of \cite{rump}). This is not the case for skew braces. 
 As an application of \cite{Shalev} we get the following:
\begin{corollary} \label{131}
Let $A$ be a skew  brace of size $p^n$, whose additive group admits an automorphism of order $p^k$ with $p^m$ fixed points. Then  the derived length of $(A, \circ )$ is bounded by $f(p^{k}, p^{m})$ with an explicit function $f$ given in \cite{Shalev}. Namely
$f(p^{k}, p^{m})=2mp^{k}$ max$\{2^{p^{k}-1}+$ log$_{2}(k+1),$  log$_{2}(m+1)\}+mp^{k}+ $ log$_{2}(mp^{k})$.
\end{corollary}

%Recall also that it was shown by Rump that adjoint groups of finite braces are soluble (see line 14 on page 147 of \cite{rump}).

Skew braces were introduced in 2017 by Guareni and Vendramin \cite{GV}. The additive group of a skew brace need not be abelian, and it is abelian only for skew braces which are braces. The above results also hold if we apply them to multiplicative groups of braces instead of additive groups. The reason for considering additive groups in Corollary \ref{121} and Corollary \ref{131} is that it is possible to apply them for the  maps $\lambda _{a}(b)=-a+a\circ b$ for $a,b\in A$, where $A$ is a skew brace. It is known that, for each $a\in A$, $\lambda _{a}$ is an automorphism of the additive group of the skew brace $A$ (see Corollary 1.10 in \cite{GV}). Note that  $\lambda _{a}(b)$ gives the first coordinate of $r(a,b)$ where $r$ is the set-theoretic solution of the Yang-Baxter equation associated to the skew brace $A$ (see \cite{GV}).

%In the context of Hopf-Galois extensions fixed points of group actions on group rings are used for constructing the Hopf algebras %associated to given Hopf-Galois extensions, for example to Hopf-Galois extensions obtained from skew braces.
%Note however that it does not seem to be immediately related to the above results.

%In the context of braces, the above results could be applied, for example, to the inner automorphisms
% $f_{a}(c)=a^{-1}\circ c\circ a$ of the multiplicative group of a brace $A$ to provide information about elements which %commute with a given element $a\in A$.  Notice that a fruitful discussion about centralizers of elements  in braces and skew %braces appears in \cite{Facchini}.

An an immediate consequence of results from \cite{Shalev} we obtain:

\begin{corollary}
Let $k, n$ be
 natural numbers and let $p$ be a prime number.  Let $(A, +, \circ )$ be a brace of cardinality $p^{n}$, and let $a\in A$ be such that the product $a\circ a \circ \cdots \circ a$ of $p^{k}$ copies of $a$ is zero (in other words $a$ has order $p^{k}$ in the group $(A, \circ )$). Let $C(a)=\{r\in A: r \circ c=c \circ r\}$
 then $C(a)$ is a subgroup of $(A, \circ)$ of cardinality $p^{m}$ for some $m$. Then the derived lenght of $(A, \circ )$ is bounded by $f(p^{k}, p^{m})$ with an explicit function $f$ given in \cite{Shalev}. Namely
$f(p^{k}, p^{m})=2mp^{k}$ max$\{2^{p^{k}-1}+$ log$_{2}(k+1),$  log$_{2}(m+1)\}+mp^{k}+ $ log$_{2}(mp^{k})$.
\end{corollary}

\begin{question} What can be said about the structure braces whose multiplicative groups have small derived length?
\end{question}

%The following result, Corollary \ref{Cor17} (ii), can be applied to automorphisms $\lambda _{a}$ of additive groups of braces, %since their additive groups are abelian and hence powerful.

In \cite{Shalev} powerful $p$-groups are applied in the study of groups admitting an automorphism with few fixed points.
On p.274 of \cite{Shalev} it is shown that a finite $p$-group with an automorphism of order $p^k$ with $p^m$ fixed points has characteristic rank $r \le m p^k$; this means that, for every characteristic subgroup $H$ of $G$ we have $d(H) \le m p^k$, where $d(H)$ is the minimal number of generators of the group $H$.

\begin{corollary}\label{Cor17}

Let $G$ be a finite $p$-group admitting an automorphism $\alpha$ of order $p^k$ with $p^m$ fixed points.
Then

(i) $d(G) \le  m p^k$.

(ii) Suppose $G$ is powerful.
Then $G$ is a product of at most $m p^k$ cyclic subgroups.

\end{corollary}

\begin{proof}

Part (i) follows immediately.
Part (ii) follows from part (i) since a powerful $p$-group generated by $d$ elements is a product of $d$ cyclic subgroups,
see \cite[Corollary 2.8]{DDMS}.
\end{proof}

Notice that Corollary \ref{Cor17} (ii), can be applied to automorphisms $\lambda _{a}$ of additive groups of braces, since their additive groups are abelian and hence powerful.
 
%Corollary \ref{Cor17} (i) can be applied to automorphisms $\lambda _{a}$ of skew braces to give bounds of the characteristic %rank of the additive groups of skew braces. This can be then used to show that additive groups of abelian braces have large %powerful subgroups, under some  additional assumptions on the element $a\in A$.

Recall that F. Catino an R. Rizo introduced circle algebras, and W. Rump introduced $\mathbb F$-braces, which are linear spaces over a field $\mathbb F$.
 Let $\mathbb F_{p}$ denote the field of $p$ elements, where $p$ is a prime number. 
 It is known that a finite brace $A$ is an $\mathbb F_{p}$-brace if and only if the additive group of $A$ is the 
 direct sum of some numbers of copies of the cyclic group of cardinality $p$. Notice that an $\mathbb F_{p}$-brace $(A, +, \circ )$ of cardinality $p^{n}$ is a direct sum of $n$ cyclic subgroup of cardinality $p$. 
 
\begin{proposition}\label{16} Let $p$ be a prime number, and $n$ and $k$ be natural numbers. 
Let $(A, +, \circ )$ be an $\mathbb F_{p}$- brace of cardinality $p^{n}$, and let $a\in A, a\neq 0$ be such that 
the product $a\circ a \circ \cdots \circ a$ of $p$ copies of $a$ is zero
 (in other words $a$ has order $p$ in the group $(A, \circ )$).  
 Then there are at least $p^{\frac n{p}}$ elements $b\in A$ such that $a*b=0$.

\end{proposition}\label{16}
\begin{proof} If $a*b=0$ for all $b\in A$ then the result holds, so suppose that there is $b\in A$ such that $a*b\neq 0$. Observe that then  $\lambda _{a}^{k}(b)\neq b$ for $k<p$,
 since $\lambda _{b}(a)=a+k a*b+p$ where $p$ is a linear combination of elements $ a*(a*\cdot  (a* b)\cdots )$. Notice that for $0<k<p$ we have $ka*b\neq -p$, as otherwise by multyplying several times from the left by $a$ we get $ka*b=0$ which is impossible since $k$ is not divisible by $p$.
Let  $S_{a}=\{r\in A: a \circ r=a+r\}=\{r\in A: a * r=0\}$. 
 Notice that $S_{a}$ is a subgroup of the additive group of $A$, so it has $p^{m}$ elements by the Lagrange theorem. 
Indeed if $x,y\in S_{a}$ then $a*x=0$ and $a*y=0$ hence $a*(x-y)=0$ hence $x-y\in S_{a}$. It suffices to show that $m\geq  {\frac n{p}}$. 
By applying Corollary \ref{Cor17} (ii)  to automorphisms $\lambda _{a}$ of additive groups of braces we obtain that $n\leq mp$ (since $\lambda _{a}^{p}$ is the identity map). It follows since additive groups of braces are abelian and hence powerful and since $S_{a}=\{r\in A: \lambda _{a}(r)=r\}$ is the set of fixed points of the automorphism $\lambda _{a}:(A,+)\rightarrow (A,+)$, where $\lambda _{a}(b)=a*b+b=a\circ b-a$.
\end{proof}

 Observe that Proposition \ref{16} gives a lower bound on the cardinality of the set $Fix(a)=\{b\in A: \lambda (a)(b)=0\}=\{b\in A: a*b=0\}$. Recall that  $Fix(a)$ was introduced in \cite{J}, and in \cite{Jespers}  it was used to investigate the length and the weight of braces and skew braces.

\section{On braces whose adjoint groups are powerful}\label{11}

Let $(A, +, \circ)$ be a brace. The group $(A, \circ )$ is called the multiplicative group of $A$; it is also called the adjoint group of $A$.

Jespers, Cedo, Okni{\' n}ski and Del Rio \cite{Rio} asked which groups are adjoint groups of braces.
We study this question and relate it to the notion of powerful $p$-groups.
It is natural to ask:

\begin{question}\label{22222}
Let $A$ be a finite brace whose adjoint group is powerful. Does it follow that $A$ a strongly nilpotent brace?
\end{question}

In this section we show that the answer to this question is positive for braces of cardinality $p^{n}$ with $p>n+1$.

\begin{proposition}\label{nilpotent}
Let $A$ be a brace of cardinality $p^{n}$ for some prime number $p$ and some natural number $n<p-1$. Then $pA=\{pa:a\in A\}$ is a subbrace in $A$ whose adjoint group is powerful. Suppose that the adjoint group of the brace $A$ is powerful. Then $A$ is a strongly nilpotent brace.
\end{proposition}
\begin{proof}
By a result from \cite{passage} we have $A^{\circ p^{i}}=p^{i}A$ for every $i$. Let $a\in A$, then
\[p^{i}a=a_{1}^{\circ p^{i}}\circ \cdots \circ a_{j}^{\circ p^{i}},\]
for some $j$ and some $a_{1}, \ldots , a_{j}\in A$.
 Observe that since $(A, \circ )$ is a  powerful group then
\[ a^{-1}\circ b^{-1}\circ a \circ b\in A^{\circ p}=pA,\]
for all $a,b\in A$, where $a^{-1}, b^{-1}$ are the inverses of $a$ and $b$ in the group $(A, \circ )$.
Therefore $a\circ b=b\circ a \circ c$ for some $c\in pA$. Consequently
$a*b=b*a+c+ (b*a)*c+b*c+a*c\in b*a+pA$.
 It follows that the factor brace
$A/pA$ is commutative (recall that $pA$ is an ideal in the brace $A$ by \cite{passage}).

It is known that commutative braces are Jacobson radical rings, and since the brace $A/pA$ is finite it follows that $A/pA$ is a nilpotent ring (where $*$ is the multiplication in the ring, where $a \circ b=a*b+a+b$ in the brace).
% Therefore $(A/pA)^{(j)}=0$ for some $j$.
Consequently, $A^{(j)}\in pA$ for some $j$.

We will now show, by induction on $m$, that
\[A^{(mj)}\subseteq p^{m}A,\]
 for every $m$.  Notice that $p^{n}A=0$, so this implies $A^{(nj)}=0$.
The result is true for $m=1$.

Suppose that the result is true for some $m$, so $A^{(mj)}\subseteq p^{m}A$. We need to show that \[A^{((m+1)j)}\subseteq p^{m+1}A.\]
We will show, by induction on $k \ge 1$, that
$A^{(mj+k-1)}\subseteq p^{m}A^{(k)}$ for all $k$. Recall that $A^{(k)}$ is the additive subgroup of $A$ generated by elements $a*c$ where $a\in A^{(k-1)}$ and $c\in A$. For $k=1$ the result is true since  $A^{(mj)}\subseteq p^{m}A$ by the above, and since $A^{(1)}=A$.

Suppose that this result is true for some $k\geq 1$ so $A^{(mj+k-1)}\subseteq p^{m}A^{(k)}$. We need to show that
\[A^{(mj+k)}\subseteq p^{m}A^{(k+1)}.\]
Let $a\in A^{(mj+k)}$. Then $a=\sum_{i=1}^{t}a_{i}*c_{i}$ for some $a_{i}\in A^{(mj+k-1)}$ and $c_{i}\in A$.
By the inductive assumption on $k$ we have \[a_{i}\in p^{m}A^{(k)}.\]
 
By Lemma  $15$ \cite{passage}, applied for the brace ${\bar A}=A^{(k)}$, we have $p^{m}{\bar A}={\bar A}^{\circ p^{m}}$, hence  \[a_{i}=b_{1}^{\circ p^{m}}\circ \cdots \circ b_{l}^{\circ p^{m}}\] for some natural number $l$ and some $b_{1}, \ldots , b_{l}\in {\bar A}=A^{(k)}$.
Observe that \[a_{i}*c_{i}=\lambda _{a_{i}}(c_{i})-c_{i}=\lambda _{b_{1}^{\circ p^{m}}\circ \cdots \circ b_{l}^{\circ p^{m}}}(c_{i})-c_{i}=\]
\[=\lambda _{b_{1}^{\circ p^{m}}}(\lambda_{{ b_{2}}^{\circ p^{m}}}(\cdots  \lambda _{b_{l}^{\circ p^{m}}}(c_{i})\cdots ))-c_{i}.\]

By Lemma $15$ \cite{passage}  and because $A^{p-1}=0$ we have
$\lambda _{{b_{1}^{\circ p^{m}}}}(c_{i})=(\sum_{j=1}^{p-1} {p^{m} \choose j} e_{j})+c_{i}$, where
$e_{1}=b_{1}*c_{i}$ and $e_{j+1}=b_{1}*e_{j}$. Notice that for a positive integer $j<p$ we have that ${p^{m} \choose j}$ is divisible by $p^{m}$. Moreover $e_{j+1}\in b_{1}*e_{j}\subseteq b_{1}*A\subseteq A^{(k)}* A\subseteq A^{(k+1)}$. The same argument works if we take $b_{2}, \cdots , b_{l}$ instead of $b_{1}$.
Therefore, $a_{i}*c_{i}\in p^{m}A^{(k+1)}$. This proves the inductive hypothesis.
Notice that this implies \[A^{(mj+j)}\subseteq p^{m}A^{(j)}.\] Recall that $A^{(j)}\subseteq pA$, therefore $A^{(mj+j)}\subseteq p^{m+1}A$. This establishes the inductive step, and completes the proof.
\end{proof}

\section{On powerful  Lie rings with left nilpotent pre-Lie
rings}\label{12}

Let $L$ be a Lie ring. We will denote the binary operations in $L$ by $+$ and $[.,. ]$.
 Let $A$ be a pre-Lie ring. We say that $A$ is a pre-Lie ring of $L$ if $A=L$ as sets, $A$  
 and $L$ have the same addition, and $[a,b]=a\cdot b-b\cdot a$ is the Lie product in $L$.

Recall that a Lie ring $L$ whose size is a $p$th power is powerful if and only if
for every $a,b\in L$,  either $p>2$ and $[a,b]\in pL$, or $p=2$ and $[a,b]\in 4 L$.

\begin{proposition}
Let $L$ be a powerful Lie ring of cardinality $p^{n}$ for some prime number $p$, and some natural number $n$. Let $(A, +, \cdot)$ be a pre-Lie ring of the ring $L$ which is left nilpotent.
Then $A$ is a right nilpotent pre-Lie ring and consequently $A$ is a (strongly) nilpotent pre-Lie ring.
\end{proposition}

\begin{proof} Notice that $p^{n}a=0$ for $a\in A$, since the additive group $(A, +)$ has cardinality $p^{n}$.
Observe first that a left nilpotent pre-Lie ring which is commutative is also right nilpotent; this follows since the commutativity of $\cdot $ implies
\[(\cdots ((a_{1}\cdot a_{2})\cdot a_{3})\cdots a_{i-1})\cdot a_{i}=a_{i}\cdot(a_{i-1} \cdots (a_{3}\cdot(a_{2}\cdot a_{1}))\cdots ).\]
Note also that $pA=\{pa:a\in A\}$ is an ideal of $A$, because the multiplication in $A$ is distributive with respect to the addition.

The factor pre-Lie ring $A/pA$ is a commutative, since $a\cdot b-b\cdot a=[a,b]\in pL=pA$, since $L$ is a powerful Lie ring.
Therefore $A/pA$ is a right nilpotent pre-Lie ring, since it is left nilpotent by the assumptions.

Therefore, for some $i$, and for all $a_{1}, \ldots , a_{i}\in A$,
\[a_{i}\cdot(a_{i-1} \cdots (a_{3}\cdot(a_{2}\cdot a_{1}))\cdots )\in pA.\]
Denote $pc= a_{i}\cdot(a_{i-1} \cdots (a_{3}\cdot(a_{2}\cdot a_{1}))\cdots )$; then \[a_{2i}\cdot(a_{2i-1} \cdots (a_{3}\cdot(a_{2}\cdot a_{1}))\cdots )=a_{2i}\cdot(a_{2i-1} \cdots (a_{i+3}\cdot(a_{i+2}\cdot (a_{i+1} \cdot pc)))\cdots )=\]
\[=p a_{2i}\cdot(a_{2i-1} \cdots (a_{i+3}\cdot(a_{i+2}\cdot (a_{i+1} \cdot c)))\cdots )\in p^{2}A.\]

Continuing in this way we obtain that $a_{ni}\cdot (a_{ni-1} \cdots (a_{3}\cdot(a_{2}\cdot a_{1}))\cdots )\in p^{n}A=0,$
for all $a_{1}, \ldots , a_{ni}$. Therefore, $A$ is a right nilpotent pre-Lie ring.
By a result from \cite{passage}, a pre-Lie ring which is left nilpotent and right nilpotent is nilpotent, so for some $m$,  products of any $m$ elements in the pre-Lie ring $A$, with any distribution of brackets, are zero.
\end{proof}

 The following question may be related to the results obtained in this section:

%\begin{question}
% Describe all braces of cardinality $p^{5}$ and $p^{6}$, where $p$ is a prime number, % whose multiplicative groups are %powerful.
% \end{question}

\begin{question} Is every finite powerful $p$-group the adjoint group of a brace?
\end{question}

Observe that, for $p>(n+1)^{n+1}$, this question can be approached by investigating powerful pre-Lie rings, following the approach described on page 141 in \cite{Rump}, where Wolfgang Rump suggested  to use Lazard's correspondence to investigate such questions. Here, we are able to add rigorous proofs to apply this method.
 Let $(A, +, \cdot )$ be a pre-Lie ring. Define $[a,b]=a\cdot  b-b\cdot a$, for $a,b\in A$. We say that $(A, +, \cdots)$ is a powerful pre-Lie ring if and only if  $(A, +, [. , . ])$ is a powerful Lie ring.

 Our next result is the following:
\begin{theorem}\label{123}
 Let $p>0$ be a prime number and $p > (n+1)^{n+1} +1$ be a natural number.
  There is one-to-one correspondence between braces of cardinality $p^{n}$ whose adjoint group is powerful and powerful  pre-Lie rings which are left nilpotent  of the same cardinality  (and with the same additive group as the additive group in the corresponding brace). 
\end{theorem}
\begin{proof} 
Let $(A, +, \circ )$ be a brace of cardinality $p^{n}$ with $p-1>(n+1)^{n+1}$. By Proposition \ref{nilpotent}, $A$ is a strongly nilpotent brace. Next, by Corollary 19 from \cite{passage2}, $A$ is strongly nilpotent of nilpotency index $k<p-1$. By Theorem 6 from  \cite{passage2} the formula \[a\cdot b=-(1+p+\ldots +p^{p})\sum_{i=0}^{p-1}\xi ^{p-1-i}((\xi ^{i}a)*b)\] produces a pre-Lie ring $(A, +, \cdot )$ from the brace $A$.

By Proposition \ref{345} this pre-Lie ring is powerful. The brace $(A, +, \circ )$ can be then obtained  by applying the construction of the group of flows to the pre-Lie ring $(A, +, \cdot)$ (see \cite{passage2}). Therefore, every brace of cardinality $p^{n}$, with $p>(n+1)^{n+1}$ and whose multiplicative group is powerful, is obtained by applying the construction of group of flows to some powerful pre-Lie algebra (which is also strongly nilpotent).

Observe on the other hand that if $(A, +, \bullet )$ is an arbitrary pre-Lie ring of cardinality $p^{n}$ then the construction of the group of flows is well defined by Lemma 9 and Theorem 10 \cite{passage}. Observe that by applying the group of flows construction to 
  the pre-Lie ring $(A, +, \bullet)$ we obtain a powerful group $G$. This follows because the group obtained from the Lie ring $(A, +, [. , .])$,  where $[a,b]=a\bullet b-b\bullet a$, by using the Lazard correspondence, is powerful (it follows from lines 7-9 on page 797 in \cite{TW}). Observe that the group  
$G$ is isomorphic to the group obtained from the pre-Lie ring $(A, +, \bullet )$ by using the 
construction of the group of flows (for more information see for example \cite{AG}, or in the context of braces see the Remark at the end of section 4 in \cite{Lazard}). Therefore, the brace obtained from pre-Lie ring $(A, +, \bullet )$, by using the construction of the group of flows has the multiplicative which is powerful group.

Reasoning as in Theorem 14 \cite{passage} we obtain that we can pair powerful left nilpotent pre-Lie rings and braces whose adjoint groups are powerful with pairs and every brace and every pre-Lie ring are exactly in one pair. This concludes the proof. 
\end{proof}

{\em Remark.} Observe that the Lie ring $(A,+, [.,.])$ obtained in Proposition \ref{345} is isomorphic to the Lie ring that resulted from applying the Lazard correspondence to the multiplicative group of the brace.  Therefore, the correspondence obtained in Theorem \ref{123}  expands upon the Lazard correspondence in the case of powerful groups.
Nilpotency of braces was investigated for example in \cite{ccs, co, col, gateva, Nilpotency, J, Jespers, pilitowska, Dora, rump, Engel}. 

 Let $A$ be a brace of cardinality $p^{n}$ for some prime $p$ and some $n$. We say that $A$ is a powerfu brace if $a\circ b-b\circ a\in pA$, for all $a,b\in A$.
 We don't know if adjoint groups of powerful braces are always powerful.
\section{Conclusions and open questions}\label{13}

   It is known that graded Lie algebras are normed (we recall the definition of normed algebras later in this section). This brings 
  several open questions about left normed pre-Lie algebras.

 \begin{question} 
  If $P$ is a left nilpotent  pre-Lie algebra whose associated  
  Lie algebra is powerful, does it follow that  $P$ is  left normed?  
 \end{question}

%In the context of braces the following more restricted question is also relevant:
%
% \begin{question} 
%  If $P$ is a finite left nilpotent  pre-Lie algebra whose associated  
%  Lie algebra is powerful, does it follow that  $P$ is  left normed?  
% \end{question}
%
 %In this section we give some thoughts related to the following question:
%
%\begin{question}
% Give a bound of a  number of graded left normed  algebras  of given cardinality.
%\end{question}

  In this section we focus not only on graded Lie algebras $L$ associated with 
 nilpotent groups, but on a larger class of objects, namely on general
 algebras $A$, where the product need not be associative, and need not satisfy 
 the Jacobi identity or any other identity.

 By a graded algebra we mean an algebra graded over the positive integers. Clearly, graded     
 algebras don't contain a unit element $1$.

Consider a finite-dimensional graded algebra $(A,+, \cdot , 0)$ over a field $F$.
Then $A =  A_1 \oplus  \ldots \oplus A_c$ for some $c \ge 1$ and $A_i \cdot A_j \subseteq  A_{i+j}$ for all $i,j \ge 1$, where we define $A_k = 0$ for all $k > c$.

We shall study below the important case where the graded algebra $A$ is generated by $A_1$. Notice that in a strongly nilpotent brace  the multiplications of generators of the additive group  will contain all  of the information about the multiplicative group of the brace. This  gives an upper bound of the number of such braces. See \cite{leandronew} for interesting results on ennumerating braces and $L$-algebras.

Define $d_i = \dim A_i$ $(i = 1, \ldots , c)$, and  
for $1 \le i \le j \le c-1$ define the product functions
\[
f_{ij}:A_i \times A_j \to A_{i+j}, \; f(x,y) = x \cdot y.
\]
These functions are bi-linear over $F$ and determine the product in $A$.
Note also that the number of bi-linear functions $f_{ij}$ is 
bounded above by $|A_{i+j}|^{d_i \cdot  d_j}$ provided $i+j \le c$, and is $1$ otherwise (since
$f_{ij}$ is the zero function in this case).

For a positive integer $m$ denote by $A^{ \{ m \} }$ the set of all products of elements $a_1, \ldots , a_m \in A$ in all possible brackets arrangements. We say that $A$ is {\em strongly nilpotent} if $A^{ \{ m \} } = 0$ for some $m \ge 1$.

 Clearly, graded algebras $A$ as above (generated by $A_1$) are strongly nilpotent. Indeed, since products are bi-linear, to show
that $A^{ \{ m \} } = 0$ it suffices to show that products of homogeneous elements $a_1, \ldots , a_m$ in all brackets arrangements vanish.

%Suppose $a_j \in A_{i_j}$ ($i = 1, \ldots , c$). Then (in all
%possible brackets arrangements) $a_1 \cdots a_m$ belongs to $A_{i_1 + \ldots + i_m} = 0$
%provided $i_1 + \ldots + i_m > c$. We conclude that, for $A$ as above, $A^{ \{ c+1 \} }=0$.

For elements $a_1, \ldots , a_m \in A$ let $a_1 \cdot a_2 \cdots a_m$
denote the left-normed product defined inductively by 
\[
a_1 \cdots a_m := a_1 \cdot ( a_2 \cdots a_m ). 
\]
These are analogs of left-normed 
commutators in groups and of left-normed Lie products in Lie algebras.

  An algebra as above (i.e., graded and generated as an $\mathbb F$-algebra  by $A_{1}$) is {\em left normed} if $A_{i}$ is spanned by left normed products. 
 It is known that Lie algebras are left normed, see the proof of the next Proposition for details.
\begin{remark}\label{left-normed}
Let $A$ be an arbitrary not necessarily associative algebra of finite dimension
over a field $F$. If $A$ is left normed then each $A_i$ is spanned by the left-normed products 
$a_{j_1} \cdots a_{j_i}$ where $a_{j_k} \in A_1$ for all $k = 1, \ldots , i$.
\end{remark}
\begin{proof}
Let $L$ be a Lie algebra. It is well-known that, for all $k \ge 1$, $L^k$ is generated as a linear space by all left-normed Lie products of weight $k$ in $L$. 
 It  follows that  if, for a group $G$, $L = L(G) = \oplus_{i=1}^c L_i$, then $L^k$ is spanned by left-normed Lie products $[x_1, \ldots , x_k]$ where $x_1, \ldots , x_k \in L_1$. 

The proof of this result uses only the fact that $L^k$ is spanned by 
left-normed Lie products of elements of $L_1$. Since, by the assumptions of Remark \ref{left-normed}, the algebra $A$ is generated by $A_1$, the result follows.
\end{proof}

%{\bf Definition.} Let $m,n$ be natural numbers. Let $S_{n,m}$ be a set of some all possible left normed algebras 
%$A = \oplus_{i \ge 1} A_i$ of  dimension $n$ over a field $\mathbb F$, with $A_{1}$ of dimension $m$. For each such algebra %there is a way in which its not left normed products can be written as linear combinations of  left normed elements. For each %algebra $A\in S_{n,m}$ we choose  one such way, and put it (i.e., this way of presenting elements as linear combinations of left %normed elements) as an element of a set which we call  $T_{n,m}$. Let $D_{m,n}$ be the cardinality of  the set $T_{n,m}$.
% It is natural to ask:
%
%$ $
%
% 
%\begin{question}
% Consider the class of graded pre-Lie algebras which are left normed and whose Lie algebras are powerful.
% Is there a good bound for  $D_{m,n}$ for all (or some)  $m,n$?
%\end{question}
%
%$ $
%
 Recall that previously we defined $A = \oplus_{i \ge 1} A_i$,  $d_i = \dim A_i$ $(i = 1, \ldots , c)$, and  
for $1 \le i \le j \le c-1$ define the product functions
\[
f_{ij}:A_i \times A_j \to A_{i+j}, \; f(x,y) = x \cdot y.
\]
 For $i = 1, \ldots , c-1$ set $f_i:= f_{1i}$.
 Consider the functions the functions $f_i$ $(1 \le i \le c-1)$.  There are at most $|A_{i+1}|^{d_1\cdot d_i}$ such functions
for each $i$. 
% Removed the following sequence as it contained D_{n}. By Corollary $3.2.4$ in \cite{Khukhro} we obtain that, since $A$ is %generated by 
% $A_1$, it is d
%etermined as an algebra by the functions $f_i$ $(1 \le i \le c-1)$ and by an element of set $T_{m,n}$.  
% Hence the number of such algebras $A$ (not just up to isomorphism but also as labelled algebras) is bounded above by 
%%$D_{m,n}$ multiplied by:
%\[
%\prod_{i=1}^{c-1} |A_{i+1}|^{d_1 \cdot d_i}  =
%\prod_{i=1}^{c-1} (|A_{i+1}|^{d_i})^{d_{1}}\]
%Let $n = \dim_F A$. This in turn is bounded above by
%\[(\prod_{i=2}^c |A|^{d_i})^{d_{1}} =
%|A|^{\dim L\cdot d_{1}} = (q^n)^{nd_{1}} = q^{n^2d_{1}}.
%\]
 % Notice that $D_{m,n}$ will depend of the cardinalities of sets $A_{i}$. So algebras $A$ with different dimensions of sets %$A_{i}$ will be counted separately.

%I removed info about D_{n} and the following resultsWe have proved part  the following result.

%{\bf I have to remove part of Theorem $2$ as $D_{n,m}$ complicated matters for lower bounds}
%\begin{theorem}\label{Lie-count} 
%
%(i) The number of labelled graded finite not necessarily associative  left normed algebras 
%$A = \oplus_{i \ge 1} A_i$ of  dimension $n$ over a field $\mathbb F$,  
%which are generated by $A_1$ is bounded above by $|A|^{\dim_F A \cdot m} \cdot D_{n,m} = a^{nm} \cdot D_{n,m}$, %where $m$ is the dimension of $A_{1}$, and $a$ is the cardinality of $A$.
%
%(ii) There are at most $p^{n^2\cdot m}\cdot D_{n, m}$ labelled graded pre Lie-rings   
%$A$ of dimension $n$ over $\F_p$ which are generated by $A_1$, where $m$ is the dimension of $A_{1}$.
%
%\end{theorem}

\medskip

 Next we recall classical results enumerating finite $p$-groups. 
We would like to comparte it with the namber of graded left normed algebras of the same cardinality. 

For large $n$, the number of groups $G$ of order $p^n$ \cite{Sims} is much smaller than $p^{(2/27+O(n^{-1/3})) n^3}$, which is Sims' upper bound for the number of groups $G$ of order $p^n$ \cite{Sims}. In Higman's paper \cite{H2} it is shown that the number of groups of order $p^n$ is at least $p^{(2/27 - o_n(1))) n^3}$.

To make this more concrete, let us restrict to groups of order $p^n$, where $n$ is sufficiently large.
A finite non-abelian $p$-group $G$ is said to be of $\Phi$-{\em class two} if it has an elementary abelian central subgroup $N$ such that $G/N$ is also elementary abelian. In particular, such groups have nilpotency class two.
A result of Higman \cite[Thereom 2.3]{H2} shows that the number of isomorphism types of $\Phi$-class two groups of order $p^n$ where $n > 1$ is at least $p^{{2 \over 27}(n^3 - 6n^2)}$. 
%Removed the following proposition, as it just repeats the previous infoThis immediately implies the following.
%
%\begin{proposition}\label{H2}
%
%The number of isomorphism types of groups of $\Phi$-class two and order $p^n$ ($n>1$) is bounded below by $p^{(2/27 - %o_n(1))n^3}$.
%\end{proposition}

\begin{question}
 Is the number of   isomorphism types of groups of $\Phi$-class two and order $p^n$ ($n>1$) larger than the number of graded Lie rings of nilpotency class two?
\end{question}
\vspace{5mm}

{\bf Acknowledgements.} The first author acknowledges support by the ISF grant 700/21, the BSF grant 2020/037 and the Vinik Chair
of mathematics which he holds.
The second author acknowledges support from the EPSRC programme grant EP/R034826/1 and from the EPSRC research grant EP/V008129/1.
Both authors are grateful to Leandro Vendramin, Wolfgang Rump for  their useful comments. We are deeply saddened by the recent passing of Avinoam Mann. We would like to express deep admiration for his wonderful results. We are also very grateful for his helpful comments to the first version of this paper.

\end{document}